\documentclass{amsart}
\usepackage{amsmath}
  \usepackage{paralist}
  \usepackage{amssymb} %% added by the authors
  \usepackage{graphics} %% add this and next lines if pictures should be in esp format
  \usepackage{epsfig} %For pictures: screened artwork should be set up with an 85 or 100 line screen
 \usepackage[colorlinks=true]{hyperref}
\hypersetup{urlcolor=blue, citecolor=red}

\numberwithin{equation}{section}
\newtheorem{theorem}{Theorem}[section]
\newtheorem{corollary}[theorem]{Corollary}
\newtheorem{lemma}[theorem]{Lemma}
\newtheorem{proposition}[theorem]{Proposition}

\newtheorem{definition}[theorem]{Definition}
\newtheorem{remark}{Remark}[section]
\makeatletter
\renewcommand{\theenumi}{\roman{enumi}}

\makeatother

\def\csubA{c_{\text{A}}}

\def\csubL{c_{\text{L}}}
\def\cA{\mathcal{A}}
\def\calB{\mathcal{B}}

\def\cM{\mathcal{M}}
\def\cP{\mathcal{P}}

\def\calB{\mathcal{B}}

\def\cX{\mathcal{X}}

\def\k0{\kappa_0}

\def\cL{c_{\text{L}}}

\def\uh{\hat{u}}
\def\vh{\hat{v}}

\def\lgl{\langle}
\def\rgl{\rangle}

\def\bC{{\mathbb C}}

\def\bR{{\mathbb R}}
\def\bZ{{\mathbb Z}}

\def\lgl{\langle}
\def\rgl{\rangle}

\def\mF{\Phi}

\def\mp{{\mathfrak p}}
\def\mq{{\mathfrak q}}

\begin{document}

%\maketitle
\title[A determining form for the NSE]
{A determining form for the 2D Navier-Stokes equations - the Fourier modes case}
\author[C. Foias]{Ciprian Foias$^{1}$}
\address{$^1$Department of Mathematics\\
Texas A\&M University\\ College Station, TX 77843}
%\address{Department of Mathematics\\
%Texas A\&M University\\ College Station, TX 77843}
\author[M.S. Jolly]{Michael S. Jolly$^{2}$}
\address{$^2$Department of Mathematics\\
Indiana University\\ Bloomington, IN 47405}
\author[R. Kravchenko]{Rostyslav Kravchenko$^{3,\dagger}$}
\address{$^3$ Department of Mathematics\\
Chicago University \\Chicago, IL 60637 }
\author[E. S. Titi]{Edriss S. Titi$^{4}$}
\address{$^4$ Department of Mathematics and Department
of Mechanical and Aerospace Engineering \\
University of California \\
Irvine, California 92697\\
Also:\\
Department of Computer Science and Applied Mathematics\\
Weizmann Institute of Science\\
Rehovot, 76100, Israel.\\
Fellow of the Center of Smart Interfaces (CSI), Technische Universit\"{a}t Darmstadt, Germany}
\address{$\dagger$ corresponding author}
\email[C. Foias]{foias@math.tamu.edu}
\email[M. S. Jolly]{msjolly@indiana.edu}
\email[R. Kravchenko] {rkchenko@gmail.com}
\email[E. S. Titi] {etiti@math.uci.edu}

%\thanks{This work was supported in part by NSF grant numbers DMS-1109638, DMS-1109645, and DMS-1109784}

\date{August 25, 2012}

\begin{abstract}
The determining modes for the two-dimensional incompressible Navier-Stokes equations (NSE)
are shown to satisfy an ordinary differential equation of the form $dv/dt=F(v)$,
in the Banach space, $X$,  of all bounded continuous functions of the variable  $s\in\mathbb{R}$ with values in certain finite-dimensional linear space.  This new evolution ODE, named {\it determining form}, induces an infinite-dimensional dynamical system in the space $X$ which  is noteworthy for two reasons.
One is that $F$ is globally Lipschitz from $X$ into itself.  The other is that the long-term dynamics of the determining form contains that of the NSE; the traveling wave solutions of the determining form, i.e.,
those of the form $v(t,s)=v_0(t+s)$, correspond exactly to initial data $v_0$ that are projections of
solutions of the global attractor of the NSE onto the determining modes.
The determining form is also shown to
be dissipative; an estimate for the radius of an absorbing ball is derived
in terms of the number of determining modes and the Grashof number (a dimensionless physical parameter).  Finally, a
unified approach is outlined for an ODE satisfied by a variety of other determining parameters
such as nodal values, finite volumes, and finite elements.
\end{abstract}
\maketitle

{\it Dedicated to Professor Peter Constantin on the occasion of his 60th birthday.}

\section{Introduction}

The notion of determining modes was introduced in \cite{CFMT,FMTT,FP,JT93} as a way to
gauge the number of degrees of freedom for the Navier-Stokes equations (NSE).
 Determining modes can be described
as a (finite) subset of Fourier modes, such that for any solution $\{u(t): t\in\mathbb{R}\}$ in the global
attractor $\cA$  its projection into those modes, denoted by $\{v(t): t\in\mathbb{R}\}$, determines the complete solution. In particular, the projection of $\{u(t): t\in\mathbb{R}\}$ into the higher modes, denoted by $\{w(t): t\in\mathbb{R}\}$, is  determined uniquely by $\{v(t): t\in\mathbb{R}\}$. Thus, there is a well defined map $W: v(\cdot) \mapsto w(\cdot)$, which will play a major role in this paper. The number of determining modes is associated
with the spatial complexity
of the long-time dynamics, and hence suggests the resolution needed for numerical computations and data assimilations (see, e.g., \cite{HOTi,OTi1}).   In this sense, estimates for number of determining modes is used much like those for the dimension of the global attractor \cite{CF88,CFT,T97}.

If an inertial
manifold exists, the high wave-number  spatial modes
are enslaved by a finite set of low wave-number modes \cite{CF88,CFNTbook,FST,FSTi,T97}.  As a consequence, if the low wave-number modes
are specified at just one point in time, the high wave-number modes are determined for all time.
In the case of an inertial manifold, the low modes satisfy an ordinary differential
equation called the {\it inertial form} whose phase space is finite-dimensional.  The existence of an inertial
manifold for the NSE, however, remains open.

Since a finite set of modes, if specified for all time, can determine the solutions
to the NSE, a natural question is whether they satisfy an {\it ordinary} differential equation  in a phase space of trajectories, the dynamics of which is consistent with  those of the NSE.  In other words,
can one find an analog of the inertial form for determining modes?
We show in Section \ref{detformsec} that the determining modes  $v$ for the NSE indeed satisfy
a {\it determining form}, which is an ordinary differential
equation of the form $dv/dt=F(v)$, where $F:X \to X$ is globally Lipschitz map, and $X$ is the Banach space of all bounded continuous
functions from $\mathbb{R}$ into the linear finite-dimensional space of the low wave-number modes, given in \eqref{2.5c}.
The key step is to extend beyond the global attractor
the enslavement of the high modes  mapping $W:v(\cdot) \mapsto w(\cdot)$, defined above, to  the space $X$ given in \eqref{2.5c}.  This is done in Section \ref{secExtend} by
analyzing a differential equation satisfied by $w(\cdot)$, and establishing the Lipschitz property of $W$.
We then show in Section \ref{absballsec} that the resulting determining form, $dv/dt=F(v)$, is
dissipative, and estimate  the radius of its absorbing ball in terms
of the number of determining modes, and the Grashof number (a dimensionless physical parameter given in \eqref{Grashof}).

Though the number of determining modes is finite, the ODE
they satisfy, i.e. the determining form, induces an infinite-dimensional dynamical system, since it is an evolution ODE  in the infinite-dimensional Banach  space $X$.   Just as the set of partial differential equations that comprise the incompressible NSE can be written as an evolution equation, involving unbounded operators, in
a Hilbert space $H$ of functions in the spatial variable, the determining
modes evolve as solutions to an ODE in a Banach space $X$ of functions in an auxiliary time variable $s$.  A significant difference, however, is that the vector field that governs the evolution in the determining form ODE is a globally Lipschitz map from the space $X$ into itself, whereas in the case of  the NSE evolution equation it involves unbounded operators.  Rather than being a dimension reduction of the NSE, the determining
form is in fact an embedding of the original dynamics in a larger phase space.
In the phase space $X$, the solutions at any moment in time $t$,
represent complete trajectories $v(t,s), \  -\infty < s < \infty$.

For certain initial trajectories the behavior of solutions to the determining form
is simple to describe.  We show in Section \ref{travwavesec} that traveling wave solutions
to the determining form can be characterized as those satisfying
$v(t,s)=\mp(t+s;u_0), \  -\infty < s < \infty$,
where $\mp(s;u_0)$ denotes the projection {\bf into} the determining modes of
the solution at time $s$ to the NSE with initial condition $u_0 \in \cA$.
As shown in Section \ref{detformsec} any solution with an initial trajectory
$v(0,s), \  -\infty < s < \infty$ which is periodic in $s$, with period $T$, remains periodic in $s$,
with period $T$, regardless of whether this initial trajectory is on the
global attractor of the NSE.
As a special case, any solution with an initial trajectory independent of $s$ remains so; if $v(0,s)=c(0), \  -\infty < s < \infty$,
then $v(t,s)=c(t), \  -\infty < s < \infty$.

We also compare the general long time behavior of the determining form with that of
the NSE.   In Section \ref{weaksec} we consider the sequences $v_n(t,s)=v(t+t_n,s-(t+t_n))$ and
$w_n(t,s)=w(t+t_n,s-(t+t_n))$.
We show that $v_n+w_n$ converges weakly to a solution of the NSE with an additional
Reynolds force which has only finitely many modes.   The stationary solutions of
the determining form satisfy a differential algebraic system: an algebraic relation for $v$,
coupled with a differential equation for $w$.   Analogs of the energy and enstrophy balances
are used in Section \ref{statsec} to show that, as in the case of the NSE,
a certain extremal condition on norms of the stationary solution occur only for
forces that are eigenfunctions of the Stokes operator.

Over the years it was proven that there are several other finite systems of parameters, different than the modes considered here, which also determine the long-time behavior of the solutions to the Navier-Stokes equations (e.g., determining nodes, finite volumes averages, finite elements, etc. \cite{CJTi1,FMRT,FT1,FT2,JT93}).  It is worth stressing that the determining form studied here uses only determining modes, on which the partial differential operators in the NSE act in an obvious manner; this action is seemingly quite involved for other systems of determining parameters. One also faces a similar problem in representing  inertial manifolds in terms of determining parameters other than the Fourier modes. This challenging general parametrization issue of the inertial manifolds has been solved in \cite{FTi} for the determining nodes case and later in \cite{CJTi1} for the general interpolant case, including the finite elements case etc. In a forthcoming paper \cite{FJKrT2} we will introduce a general, unified construction of determining forms based on any finite system of determining parameters. This unified approach has also been used in \cite{AOTi,ATi} in the context of  continuous data assimilations and finite-dimensional feedback control, respectively.  In section 11 of this paper we briefly illustrate this general construction for the finite system of determining modes considered in this paper.

\section{Preliminaries on the  Navier-Stokes equations} \label{prelim1}

We write the incompressible Navier-Stokes equations
\begin{align}
&\frac{\partial u}{\partial t} -
\nu \Delta u  + (u\cdot\nabla)u + \nabla p = \mF, \\
&\text {div} u = 0 \\
& \int_{\Omega} u\,  dx =0 \;,\qquad \int_{\Omega} \mF\,  dx =0 \\
& u(t_0,x) = u_0(x)
\end{align}
subject to periodic boundary conditions with basic domain $\Omega =[0,L]^2$ as
a differential equation in a certain Hilbert space $H$ (see \cite{CF88} or \cite{T97}),
\begin{equation}\label{NSE}
\begin{aligned}
&\frac{d}{dt}u(t) + \nu Au(t) + B(u(t),u(t)) = f,\\ &u(t) \in H\;,\
t \ge 0\;, \ \text{and} \ u(0)=u_0 \;.
\end{aligned}
\end{equation}
The phase space $H$ is the closure in $L^2(\Omega)^2$
of the set of all  $\mathbb{R}^2$-valued trigonometric polynomials $\phi$ such that
$$
\nabla \cdot \phi = 0, \qquad \text{ and} \quad \int_{\Omega} \phi(x) dx =0.
$$
The bilinear operator $B$ and force $f$ are defined as
\begin{equation}
\label{nabla1}
B(u,v)=\cP\left( (u \cdot \nabla) v \right)\;,\quad f=\cP \mF\;,
\end{equation}
where $\cP$ is the Helmholtz-Leray orthogonal projector of
$L^2(\Omega)^2$ onto $H$, and where $u,v$ are sufficiently smooth.
When the input to $B$ is lengthy, and repeated,
for convenience we will use the shorthand notation
\begin{equation}\label{shorthand}
B(u)=B(u,u)\;.
\end{equation}
Unless specified otherwise,  $f \in H$.
The scalar product in $H$ is taken to be
$$
(u,v)=\int_{\Omega}u(x)\cdot v(x) dx, \quad \text{where} \quad a\cdot b=a_1b_1+a_2b_2,
$$
with associated $L^2$-norm
\begin{equation}\label{L2norm}
|u|=(u,u)^{1/2}=\left(\int_{\Omega} |u(x)| dx\right)^{1/2}\;,
\end{equation}
where, for a vector $a\in \bC^2$, we denote its length by
$|a|=(a\cdot a^*)^{1/2}$, with $a^*=(\bar{a_1},\bar{a_2})$.
It will be clear from the context when $|\cdot|$ refers to
the length of a vector in $\bC^2$, as opposed to an $L^2$-norm.
The operator $A=-\Delta$ is self-adjoint, and its eigenvalues are of the form
$$
\left(\frac{2\pi}{L}\right)^2k\cdot k, \quad \text{where}
\ k \in \bZ^2 \setminus \{ 0\}.
$$
We denote these eigenvalues by
$$0<\lambda _0=\left(2\pi/ L\right)^2
\leq \lambda _1\leq\lambda_2 \le \cdots $$
arranged in increasing order and counted according to their multiplicities,
and write $w_0, w_1, w_2,  \ldots$, for the corresponding normalized
eigenvectors (i.e., $|w_j|=1$ and $Aw_j=\lambda_j w_j$ for $j=0,1,2,\ldots$).

The positive roots of the operator $A$ are defined by linearity from
$$
A^{\alpha}w_j= \lambda_j^{\alpha} w_j, \quad \text{for} \ j=0,1,2, \ldots
$$
on the domain
$$
D(A^{\alpha})=\{u \in H: \sum_{j=0}^{\infty} \lambda_j^{2\alpha}(u,w_j)^2 < \infty\},
$$
for $\alpha \ge 0$.

We take the natural norm on the space $V:=D(A^{1/2})$ to be
\begin{equation}\label{Vnorm}
\|u\|=|A^{1/2}u|=\left(\int_{\Omega}\sum_{j=1}^{2}
\frac{\partial}{\partial x_j}u(x)\cdot
             \frac{\partial}{\partial x_j}u(x)  dx\right)^{1/2}
=\left(\sum_{j=0}^{\infty} \lambda_j(u,w_j)^2 \right)^2.
\end{equation}
Since the boundary conditions are periodic,
we can express any element in $H$ as a Fourier series
\begin{equation}\label{Fourier}
u(x)=\sum_{k\in \mathbb{Z}^2}\uh_k e^{i\kappa_0 k\cdot x}\;,
\end{equation}
where
\begin{equation}\label{k0def}
\kappa_0=\lambda_0^{1/2}=\frac{2\pi}{L}\;, \quad \uh_0=0\;, \quad
\uh_k^*=\uh_{-k}\;,
\end{equation}
and due to incompressibility,
$k \cdot \uh_k=0$.   We associate to each term in  \eqref{Fourier}
a {\it wave number} $\kappa_0|k|$.  Parseval's identity reads as
$$
|u|^2=L^2\sum_{k\in \bZ^2} \uh_k\cdot \uh_{-k} = L^2\sum_{k\in \bZ^2} |\uh_k|^2,
$$
as well as
$$
(u,v)=L^2\sum_{k\in \bZ^2} \uh_k\cdot \vh_{-k}\;,
$$
for $v=\sum \vh_k e^{i\kappa_0 k\cdot x}$.  We define the orthogonal
projectors $P_N,Q_N$,
$P_N : H\to \ \hbox{span}\{w_j|\lambda _j\leq (\kappa_0{N})^2 \}$,  by
$$
P_{N}u =\sum_{|k| \le N} \uh_ke^{i\kappa_0 k\cdot x}\;,
$$
where $u$ has the expansion in \eqref{Fourier}, and $Q_{N}=I-P_{N}$.

Recall the orthogonality relations of the bilinear term
(see, e.g., \cite{CF88,T97})
\begin{equation}\label{flip}
\lgl B(u,v),w \rgl = -\lgl B(u,w),v\rgl\;, \quad u,v,w \in V\;,
\end{equation}
where $\lgl \cdot, \cdot \rgl$ denotes the the action of an element in $V'$,
as well as
\begin{equation}\label{ortho}
(B(u,u),Au) = 0\;,\quad u \in D(A)\;.
\end{equation}
The strengthened form of the enstrophy invariance
\begin{equation}\label{1.7.5}
(B(Av,v),u)=(B(u,v),Av)\;,\quad u\in V\;, \quad v \in D(A)\;
\end{equation}
is proved in \cite{DFJ4}.  Relation \eqref{1.7.5} together with \eqref{flip} implies
%(see e.g. \cite{FJMR})
\begin{equation}\label{moveu}
(B(v,v),Au) + (B(v,u),Av) + (B(u,v),Av) = 0\;,\quad u,v \in D(A)\;,
\end{equation}
\begin{equation}\label{1.8.5}
(B(Av,v),u) - (B(v,Av),u) + (B(v,v),Au) = 0\;,\quad u,v \in D(A)\;,
\end{equation}
see \cite{FJMR}.
We will also use Agmon's inequality
\begin{equation}\label{Agmon}
\|u\|_{\infty} \leq \csubA|u|^{1/2} |Au|^{1/2}\;, \quad u \in D(A)\;,
\end{equation}
and one of its consequences
\begin{equation}\label{otroB-est}
|(B(u,v),w)| \le \csubA |u|\|v\|^{1/2}|A^{3/2}v|^{1/2}|w|\;, \quad u \in H, v \in D(A^{3/2}), w \in H\;.
\end{equation}
%We will also refer to the constants from the Brezis-Gallouet inequality
%\begin{equation}\label{brezis}
%\|u\|_{\infty} \leq \cB\|u\|\left ( \log\frac{|Au|}{\kappa_0 \| u \|} +
%1\right )^{1/2}\;,\quad u \in D(A)
%\end{equation}
%as well as Ladyzhenskaya's
%\begin{equation}\label{Lady}
%|u|_{L^4(\Omega)}^2 \le \csubL |u|\|u\|\;, \quad u \in V\;.
%\end{equation}
(See \cite{CF88,TemamNSEbook, DG}).

%The inequalities are used to estimate the nonlinear term.
%A direct consequence of Ladyzhenskaya's inequality is
%\begin{equation}\label{LadyB}
%|(B(u,v),w)| \le \cL |u|^{1/2}\|u\|^{1/2}\|v\| |w|^{1/2}\|w\|^{1/2}\;, \quad u,v,w  \in V\;.
%\end{equation}
%Combining Agmon's inequality and the interpolation $|A^{1/2}v| \le |v|^{1/2}|Av|^{1/2}$
%gives
%\begin{equation}\label{A.46b}
%|(B(u,v),w)| \le \cI |u|^{1/2}|Au|^{1/2}|v|^{1/2}|Av|^{1/2} |w|\;, \quad u,v \in D(A), w \in H\;.
%\end{equation}
%Greater regularity allows

In this paper we will assume $f\in D(A^{j/2})$
for some specified $j\in \{0,1,2,3\}$, which is sufficient to guarantee that
\begin{equation}\label{inprodder}
\frac{1}{2}\frac{d}{dt}|u|^2=(\frac{du}{dt},u)\;, \qquad
\frac{1}{2}\frac{d}{dt}\|u\|^2=(\frac{du}{dt},Au)
\end{equation}
in the sense of distributions
(see Lemma 1.2 in Chapter 3 in \cite{TemamNSEbook} and its proof).
Taking the scalar product of \eqref{NSE} with $u$, respectively $Au$, and applying \eqref{flip}, \eqref{ortho}, and \eqref{inprodder}  gives the energy, enstrophy balance equations
\begin{equation}\label{energyeq}
\frac{1}{2} \frac{d}{dt} |u|^2 + \nu \|u\|^2 = (f,u) \;,
\end{equation}
\begin{equation}\label{enstrophyeq}
\frac{1}{2} \frac{d}{dt} \|u\|^2 + \nu |Au|^2 = (f,Au) \;.
\end{equation}
In the scientific literature,
$$
\k0^2|u|^2 = 2 \text{ times the total {\it energy} per unit mass}\;,
$$
and
$$
\k0^2\|u\|^2= \text{the total {\it enstrophy} per unit mass}\;.
$$

Straightforward
applications of the Cauchy-Schwarz, Young, and Gronwall inequalities
to \eqref{enstrophyeq} gives
\begin{equation}\label{Gronny}
\|u(t)\|^2 \le e^{-\nu\k0^2t}\|u(0)\|^2 + \frac{|f|^2}{\nu^2\k0^2}
\left(1-e^{-\nu\k0^2t}\right) \;.
\end{equation}
If $|f|=0$, we have by \eqref{Gronny} that $\|u(t)\|\to 0$
as $t\to \infty$.  Therefore we assume throughout the paper that
\begin{equation}\label{fbarnot0}
|f|>0\;.
\end{equation}
It is clear from \eqref{Gronny} that  the closed ball $\calB$ of radius
$2\nu\k0 G$, centered at 0, is {\it absorbing}
and invariant, meaning that for any $u_0$ there exists
$T=T(\|u_0\|)$ such that
\begin{equation}\label{AbsTime}
\|u(t;u_0)\| \le 2\nu\k0 G\;, \quad \text{for all } t \ge T\;,
\end{equation}
and
\begin{equation}\label{absball}
\|u(t;u_0)\| \le 2\nu\k0 G \;, \quad \text{for all } t\ge 0\;,  u_0 \in \calB
\end{equation}
where
\begin{equation}\label{Grashof}
G=\frac{|f|}{\nu^2\k0^2}
\end{equation}
is the generalized Grashof number introduced in \cite{FMTT}.
It is well known (see e.g. \cite{FJMR}) that for the flow to be turbulent, we must have
%\begin{equation}\label{Gbig}
$G \gg 1$.

%\end{equation}
While most of our results apply for all $G>0$, we will eventually
assume $G \gg 1$.
We also have from \eqref{Gronny} that
\begin{equation}\label{attbnd}
\|u\| \le \nu\k0 G \quad \text{for all}\quad u \in \cA\;,
\end{equation}
where $\mathcal{A}$ is the global attractor, i.e. the minimal compact set in
$V$ which attracts uniformly all solutions of the Navier-Stokes equations
starting from the ball $\{h\in V: \|h\|\leq 2 G\nu\kappa_0\}$. Recall that
this set can also be characterized as
\begin{equation}\label{attractordef}
\mathcal{A}=\{u_0\in H:\exists\quad u(t,u_0)=\text{solution }\forall \ t\in\mathbb{R},\sup_t\|u(t)\|<\infty\}.
\end{equation}
We often denote the
solution as simply $u(t)$
when the choice of $u_0$ is irrelevant to our considerations.
It is easily shown using \eqref{Gronny} that $T$, the
time of absorption in \eqref{AbsTime}, can be taken to be
\begin{equation}\label{AbsTime2}
T(\|u_0\|)=\frac{1}{\nu\k0^2}\max\{1,\log \frac{\|u_0\|^2}{3\nu^2\k0^2G^2}\}\;.
\end{equation}
We will use the following bounds on the nonlinear term (see \cite{DFJ,FMRT,Ti2} and the references therein)
\begin{equation}\label{titi}
|(B(w,u),v)|\leq c_T\|w\|\|u\|\left(\ln\frac{e\|v\|}{\kappa_0|v|}\right)^{1/2}|v|,
\end{equation}
\begin{equation}\label{agmon}
|(B(w,u),v)|\leq c_B\|w\|\|u\|\left(\ln\frac{e|Aw|}{\kappa_0\|w\|}\right)^{1/2}|v|,
\end{equation}
\begin{equation}\label{agmon_1}
|(B(w,u),v)|\leq c_T\|w\|\|u\|\left(\ln\frac{e|Au|}{\kappa_0\|u\|}\right)^{1/2}|v|,
\end{equation}
\begin{equation}\label{Lad}
|(B(u,v),w)|\leq \csubL|u|^{1/2}\|u\|^{1/2}\|v\||w|^{1/2}\|w\|^{1/2},
\end{equation}
\begin{equation}\label{A46a}
|(B(u,v),w)|\leq \csubL|u|^{1/2}\|u\|^{1/2}\|v\|^{1/2}|Av|^{1/2}|w|,
\end{equation}
\begin{equation}\label{A46b}
|(B(u,v),w)|\leq \csubA|u|^{1/2}|Au|^{1/2}\|v\||w|
\end{equation}
\begin{equation}\label{A46c}
|(B(u,v),w)|\leq  \csubA|u|\|v\||w|^{1/2}|Aw|^{1/2}
\end{equation}
which are valid whenever the norms involved make sense.   One can use these inequalities to
extend the domain of $B$ to larger functional spaces.  For instance from \eqref{A46b}
one can infer that $B(u,v)$ can be extended by continuity to a map from $D(A) \times V$ to
$H$.  We will denote these extensions also by  $B$.
\section{Specific preliminaries}

We will use the following elementary estimates in several places.

\begin{lemma}\label{useit}
Let $x(t)$, $y(t)$ be nonnegative functions defined on the interval $I$, where
$I=[t_0,t_1)$, for $t_0$ finite, or $(-\infty,t_1)$. Suppose that $x^2(t)$ is differentiable on $(t_0,t_1)$, and satisfies
\begin{equation}\label{e-q1}
\frac{d}{dt}x^2(t)\leq{}-ay^2(t)+by(t),\quad{}x(t)\leq\varepsilon{}y(t)
\end{equation}
for $t\in I$ where $a,b,\varepsilon$ are positive constants. Then the
following hold:
\begin{enumerate}
\renewcommand{\theenumi}{(\roman{enumi})}
\renewcommand{\labelenumi}{\theenumi}
\item\label{item1} If $x(t)\geq{\varepsilon}b/a$ for all $t\in[t_0,t_1)$, then
\begin{equation}
x(t)-\varepsilon\frac{b}{a}\leq{}e^{-a(t-t_0)/2\varepsilon^2}(x(t_0)-\varepsilon\frac{b}{a});
\end{equation}
\item\label{item2} if $x(\tilde{t})\leq\varepsilon{b/a}$, for some $\tilde{t}\in I$, then $x(t)\leq\varepsilon{b/a}$ for
all $t \in[\tilde{t}, t_1)$;
\item\label{item3} if $x(t)$ is bounded on $(-\infty,t_1)$,  then $x(t)\leq\varepsilon{b/a}$
for all $t\in(-\infty,t_1)$.
\end{enumerate}
\end{lemma}
\begin{proof}
To prove \ref{item1} % let $x(t)\geq\varepsilon{}b/a$ for all $t\in(t_0,t_1)$, and
define $c(t):=\varepsilon^{-1}x(t)-b/a$.  Since
\begin{equation}\label{ybndbelow}
y(t)\geq{} {x(t) \over \varepsilon}=c(t)+{b \over a} \geq{} {b \over a}
\end{equation}
we have
\begin{equation*}
\begin{aligned}
& \frac{dc}{dt}\frac{2\varepsilon^2}{a}(ac+b)=2\varepsilon^2c\frac{dc}{dt}+2\varepsilon^2\frac{b}{a}\frac{dc}{dt}=\frac{d}{dt}x^2\leq{}-ay^2+by\\&\leq{}-a(c+\frac{b}{a})^2+b(c+\frac{b}{a})=-c(ac+b).
\end{aligned}
\end{equation*}
Thus, since $ac+b=ax/\varepsilon \ge b > 0$,
\begin{equation}
e^{at/2\varepsilon^2}c(t)\leq{}e^{at_0/2\varepsilon^2}c(t_0)
\end{equation}
and \ref{item1} follows.

To prove \ref{item2}, suppose that $x(t')>\varepsilon{b/a}$ for some
$t'>\tilde{t}$.  It follows that both $\frac{d}{dt}x^2>0$ and
$x>\varepsilon{b}/a$ hold at some time between
$\tilde{t}$ and $t'$.  We have at that time by \eqref{ybndbelow}
\begin{equation}
0<\frac{d}{dt}x^2\leq{}-ay^2+by\leq{0},
\end{equation}
a contradiction.

\ref{item3} follows from \ref{item1}.
%To prove \ref{item3}, suppose that $x(t')>\varepsilon{b/a}$ for some point
%$t'$.  Then by the second item, $x(t)>\varepsilon{b/a}$ for all $t\leq{t'}$.
%It follows that
%\begin{equation}
%\frac{d}{dt}x^2<-ay^2+by<-a\varepsilon^{-2}x^2(t')+b\varepsilon^{-1}x(t')<0
%\end{equation}
%and so $x(t)$ is decreasing with a fixed speed for all $t\leq{t'}$.  Since it
%is bounded, it is impossible.
\end{proof}

We also use the following immediate corollary of Lemma~\ref{useit} for the case in which
the coefficient $b$ varies with time.

\begin{corollary}\label{useit1}
Suppose $x(t)$, $y(t)$ are positive functions defined on the interval
$I$, where
$I=[t_0,t_1)$ for $t_0$ finite, or $(-\infty,t_1)$, and $\varepsilon>0$ is
 such that $x(t)\leq\varepsilon{}y(t)$ and
\begin{equation}
\frac{d}{dt}x^2(t)\leq{}-ay^2(t)+b(t)y(t)
\end{equation}
for $t\in I$, with $a,b(t)$ positive, then
\begin{enumerate}
\renewcommand{\theenumi}{(\roman{enumi})}
\renewcommand{\labelenumi}{\theenumi}
\item\label{item11} if $x(\tilde{t})\leq\varepsilon\sup_{t'\leq\tilde{t}}{b(t')/a}$, for some $\tilde{t}\in I$,  then $x(t)\leq\varepsilon\sup_{t'\leq{t}}{b(t')/a}$ for
all $t\in[\tilde{t},t_1)$
\item\label{item21} if $x(t)$ is bounded on $(-\infty, t_1)$, then $x(t)\leq\varepsilon{b/a}$ for all $t\in(-\infty,t_1)$.
\end{enumerate}
\end{corollary}

\begin{definition} \label{detmodedef}
The projection $P_N$ is called {\it determining} if (and only if) for any two solutions $u_1(t),u_2(t)\ (t\in\mathbb{R})$ in $\mathcal{A}$ we have
\begin{equation}
P_Nu_1(t)\equiv P_Nu_2(t)\Longrightarrow u_1(t)\equiv u_2(t)
\end{equation}
\end{definition}

This is equivalent to the original definition in which the high mode trajectories are determined by the limit as $t \to \infty$ of their projections \cite{FK,FMRT,FP}.  Using the original definition, it is shown in \cite{JT93} that
for $P_N$ to be determining, it is sufficient to take $N=O(G)$.  For completeness,
we now reproduce that estimate using Definition \ref{detmodedef}.  We recall (see \cite{FMRT}, Ch.$3$ Sec.$1$) that $u=u_1-u_2$ satisfies
\begin{equation}
\frac{1}{2}\frac{d}{dt}|u|^2+\nu\|u\|^2=-(B(u,u_1),u)\leq \csubL|u|\|u\|\|u_1\|\leq \csubL|u|\|u\|G\nu\kappa_0,
\end{equation}
whence (since $P_Nu=0$ and so $\kappa_0N|u|\leq\|u\|$)

\begin{equation}
\frac{d}{dt}|u|^2+2\nu\|u\|^2\leq 2\csubL\frac{G\nu\kappa_0}{\kappa_0 N}\|u\|^2
\end{equation}
\begin{equation}
\frac{d}{dt}|u|^2+2\nu N^2\lambda_0\left(1-\frac{G\csubL}{N}\right)|u|^2\leq 0
\end{equation}
so that
\begin{equation}
|u(t)|^2\leq |u(t_0)|^2 e^{-2\nu N^2\lambda_0(1-G\csubL/N)(t-t_0)},\quad\forall
\ t_0\leq t.
\end{equation}
Thus if

\begin{equation}\label{2.1}
N>\csubL{}G,
\end{equation}
we conclude (by letting $t_0\rightarrow -\infty$) that $u(t)\equiv 0$; that is if \eqref{2.1} holds, then $P_N$ is determining. Note that \eqref{2.1} implies that if $u_1(t)=u_2(t)$ only for all
 $t\leq T$ then still $u_1(t)\equiv u_2(t)$.
In particular $Q_Nu(\cdot)$ is uniquely determined by $P_Nu(\cdot)$ in case \eqref{2.1} holds.
Then by denoting

\begin{equation}\label{2.2}
\mq(s)=W(s;P_Nu(\cdot))=Q_Nu(s) \quad(s\in\mathbb{R}),
\end{equation}
we have that on the set
$\{P_Nu(\cdot):u(\cdot)\text{ is a solution of \eqref{NSE} in }\mathcal{A}\}$ the map
\begin{equation}\label{2.new}
P_Nu(\cdot)\mapsto \mq(s)=W(s;P_Nu(\cdot)) \quad(s\in\mathbb{R})
\end{equation}
 is well-defined by
the following properties
\begin{equation}\label{2.3a}
%\frac{d\mq(s)}{ds}+\nu A\mq(s)=Q_Nf-Q_N B(P_Nu(s)+\mq(s),P_Nu(s)+\mq(s)),\quad\forall s
\frac{d\mq(s)}{ds}+\nu A\mq(s)=Q_Nf-Q_NB(P_Nu(s)+\mq(s)),\quad\forall \ s\in\mathbb{R}
\end{equation}
\begin{equation}\label{2.3b}
\sup_{s\in\mathbb{R}}\|\mq(s)\|<\infty,
\end{equation}
where we used the shorthand notation in \eqref{shorthand}.

\section{Extending the domain of $W$ to $C_b(\bR,P_NV)$}\label{secExtend}

To ultimately write a differential equation for the low mode trajectories,
we need to extend the domain of $W$ beyond the projection of the global attractor
into $P_NV$.
We define
\begin{equation}\label{2.4a}
\mp(t,s)=P_Nu(s+t),\quad(\forall \ s,t\in\mathbb{R}),
\end{equation}
and apply the projection $P_N$ to \eqref{NSE} to obtain
\begin{equation}\label{2.4b}
\begin{aligned}
%& \frac{\partial}{\partial t}\mp(t,s)=-\nu A\mp(t,s)+P_Nf-\\ &P_NB(\mp(t,s)+W(t+s;P_Nu(\cdot)),\mp(t,s)+W(t+s;P_Nu(\cdot))).
\frac{\partial}{\partial t}\mp(t,s)=-\nu A\mp(t,s)+P_Nf-P_NB(\mp(t,s)+W(t+s;P_Nu(\cdot))).
\end{aligned}
\end{equation}
But it is clear that
\begin{equation}\label{2.5a}
W(t+s;P_Nu(\cdot))\equiv W(s;\tau_tP_Nu(\cdot))\equiv W(s;\mp(t,\cdot)),
\end{equation}
where $\tau_t$ is the shift by $t$:
\begin{equation}\label{shiftdef}
\tau_tu(s)=u(s+t)\;, \quad \text{for all } s \in \bR\;.
\end{equation}
Thus \eqref{2.4b} becomes
\begin{equation}\label{2.5b}
\frac{\partial}{\partial t}\mp(t,s)=-\nu A\mp(t,s)+P_Nf-P_NB(\mp(t,s)+W(s;\mp(t,\cdot))).
%\frac{\partial}{\partial t}\mp(t,s)=-\nu A\mp(t,s)+P_Nf-P_NB(\mp(t,s)+W(s;\mp(t,\cdot)),\mp(t,s)+W(s;\mp(t,\cdot))).
\end{equation}
We now consider the Banach spaces
\begin{equation}\label{2.5c}
\begin{aligned}
X&=C_b(\mathbb{R},P_NH)=\{v:\mathbb{R}\mapsto P_NH,  v(s)\text{ continuous }\forall\ s\in\mathbb{R},\|v\|_X<\infty\},\\
Y &=C_b(\mathbb{R},Q_NH)=\{w:\mathbb{R}\mapsto Q_NH,  w(s)\text{ continuous }\forall\ s\in\mathbb{R},|w|_Y<\infty\}
\end{aligned}
\end{equation}
where
\begin{equation}
\label{norm_X}
 \|{v}\|_X=\sup_s{\|{v}(s)\|} \;, \quad |w|_Y=\sup_s{|w(s)|}\;.
\end{equation}
Let $\mp(t,s)$ in \eqref{2.4a} be written as $v(t)\in{X}$, so that \eqref{2.5b}
can be viewed as the ``differential equation"
\begin{equation}\label{2.6}
\frac{d}{dt}v(t)=-\nu Av(t)+P_Nf-P_NB(v(t)+W(v(t))) \quad(\forall \ t\in\mathbb{R}),
%\frac{dv(t)}{dt}=-\nu Av(t)+P_Nf-P_NB(v(t)+W(t;v),v+W(t;v)) \quad(\forall t\in\mathbb{R}),
\end{equation}
where {\it throughout $v(t)$ denotes the function $s\mapsto{v(t,s)}$ in $X$} and $W(v(t))$ denotes the function $s\mapsto{W(s;v(t))}$. We should stress
that $dv/dt$ in \eqref{2.6} is the derivative of a {\it trajectory} in the space $X$.
It is also useful to note that, in particular, the following partial differential equation holds
$$
\frac{\partial}{\partial t}v(t,s)=P_N f - \nu Av(t,s)-P_NB(v(t,s)+W(s;v(t,\cdot)))\;,
$$
for every $s \in \bR$, where $\partial/\partial t$ is taken in $H$.
Note that \eqref{2.6} is defined only for those $v(t)$ that come from $\mp(t,s)$, i.e.
from $u \in \cA$.
%To extend \eqref{2.6} to an ordinary differential equation on the whole space $X$,

Now let $\varphi(\xi)$ be a continuous, piecewise linear function, that is equal to $1$ for
$|\xi|\leq{}2G$, $0$ for $|\xi|\geq{}3G$, where $G$ is the Grashof number given in \eqref{Grashof}, and consider
\begin{equation}
\label{ode_final}
\begin{aligned}
&\frac{d}{dt}v(t)+\nu{A}v(t)+P_NB\left(\varphi v(t)+W(v(t))\right)=P_Nf \;,
%&\frac{d}{dt}v(t)+\nu{A}v(t)\\&+P_NB\left(\varphi\left(\frac{\|v(t)\|}{\nu\kappa_0}\right)v(t)+W(t;v),\varphi\left(\frac{\|v(t)\|}{\nu\kappa_0}\right)v(t)+W(t;v)\right)=P_Nf
\end{aligned}
\end{equation}
where
\begin{equation}\label{varphidef}
\varphi=\varphi\left(\frac{\|v(t)\|}{\nu\kappa_0}\right)\;.
\end{equation}
We will show that $W(v)$ can be extended to a Lipschitz map $X\rightarrow{Y}$
in such a way that \eqref{ode_final} defines an {\it ordinary} differential
equation in $X$, which coincides with \eqref{2.6} on the ball of radius $2\nu\k0G$ in
$X$.   By this we mean that $v(t)$ in \eqref{ode_final} is the solution of
an equation
\begin{equation}\label{2.7a}
\frac{dv}{dt}=F(v) \quad(\forall\ t\in\mathbb{R},v\in X),
\end{equation}
where $F:X\to X$ is a locally Lipschitz function. In fact we will show that $F$
is globally Lipschitz
\begin{equation}\label{2.7b}
\|F(v_1)-F(v_2)\|_X\leq L_F\|v_1-v_2\|_X,
\end{equation}
where the constant $L_F$ is independent of $v_1, v_2$.
%In the next section we will define $F(t,v)$, i.e. the
%equation~\eqref{2.7a}, satisfying \eqref{2.7b}.
This will provide a proof
that equation \eqref{NSE} on the attractor $\mathcal{A}$ is transformed into an
ordinary differential equation in the
space $X$ defined in \eqref{2.5c}.
%To this aim we need to extend the domain of
%definition of the map $W$.

%We assume that
%\begin{equation}
%G\geq 1.
%\end{equation}
% We also assume that
%\begin{equation}
%\label{bound_for_N}
%\begin{aligned}
%&N\geq\alpha{G}\ln{eN}\\&\alpha>\max\{48c_T,8c_A+12c_B+12c_T+8c_L),2\}
%\end{aligned}
%\end{equation}
% It will be clear from
%what follows why we need such constraints.
Let $h\in Q_NH$, $v$ be an element of $X$.  We will show that
%and consider the following equation:
\begin{equation}\label{NSEw}
%\begin{aligned}
\frac{d}{ds}w + \nu Aw+Q_NB(\varphi v+w,\varphi v+w) = h\;.
%\end{aligned}
\end{equation}
%We will show that %for each $v$ in $X$, \eqref{NSEw}
has a unique bounded solution
$w(s)\in Y$. %{}C_b(\mathbb{R},Q_N H)$, $s\in\mathbb{R}$.
Consider a Galerkin approximation for $n>N$, of equation~\eqref{NSEw}, namely
\begin{equation}\label{NSEwG}
\begin{aligned}
& \frac{d}{ds}w_n + \nu Aw_n +P_nQ_NB(\varphi{}v+w_n,\varphi{}v+w_n) = P_nh \\ & w_n(s)\equiv P_n w_n(s)\text{ for }s\geq s_0,
\end{aligned}
\end{equation}
with the initial condition $w_n(s_0)=0$.  Since \eqref{NSEwG} is an ordinary
differential equation, it has a unique solution $w_n(s)$ for $s$ in some maximal interval
$[s_0,s_1)$, on which $\|w(s)\| < \infty$.

Taking the scalar product of \eqref{NSEwG} with $Aw_n$ and using relation~\eqref{moveu} we obtain:

\begin{equation}\label{1.1.1}
%\begin{aligned}
\frac{1}{2}\frac{d}{ds}\|w_n\|^2 + \nu |Aw_n|^2 +\varphi^2(B(v,v),Aw_n)-\varphi(B(w_n,w_n),Av) =(h,Aw_n)\;.
%\end{aligned}
\end{equation}
%where $\varphi=\varphi\left(\frac{\|v\|}{\nu\kappa_0}\right)$.
Note that it follows from the definition of $\varphi$ that
\begin{equation}\label{3.0}
\left[\varphi\left(\frac{\|v\|}{\nu\kappa_0}\right)\right]^{\alpha}
\|v\|\leq{}3\nu\kappa_0G\quad\forall\ {v}\in{X}\;,
\ \forall\ \alpha >0\;.
\end{equation}
Using \eqref{agmon_1} and \eqref{3.0}, we have
\begin{equation*}
\begin{aligned}
\varphi^2|(B(v,v),Aw_n)| \leq \varphi^2c_T\|v\|^2|Aw_n|(\ln{eN})^{1/2}
\leq{}9c_T\nu^2\kappa_0^2G^2(\ln{eN})^{1/2}|Aw_n|\;,
\end{aligned}
\end{equation*}
while using \eqref{Lad}, $\|w_n\|\leq{}(1/\kappa_0N)|Aw_n|$, $|A^{3/2}v|\leq(\k0{}N)^2\|v\|$ and \eqref{3.0}, we get
\begin{equation*}
\begin{aligned}
&\varphi|(B(w_n,w_n),Av)|=\varphi|(B(w_n,Av),w_n))|\leq \csubL |w_n|\|w_n\||A^{3/2}v|\leq \\&
3\csubL\nu {G \over N} |Aw_n|^2.
\end{aligned}
\end{equation*}
Thus
\begin{equation}\label{bound_Aw}
\begin{aligned}
\frac{1}{2}&\frac{d}{ds}\|w_n\|^2 + \nu |Aw_n|^2 \\
&\leq |h||Aw_n|+\varphi^2|(B(v,v),Aw_n)|+\varphi|(B(w_n,w_n),Av)| \\
&\le 3\csubL\nu\frac{G}{N}|Aw_n|^2+\left[\frac{|h|}{|f|}\nu^2\kappa_0^2G+9c_T\nu^2\kappa_0^2G^2(\ln{eN})^{1/2}\right]|Aw_n|\;.
\end{aligned}
\end{equation}
Now we have by \ref{item2} of Lemma~\ref{useit} (with
$\varepsilon=1/(\kappa_0N)$), that if $N$ and $G$ are such that
\begin{equation}\label{gn_condition1}
N>3\csubL{}G,
\end{equation}
then
\begin{equation}\label{bound_w_n}
\begin{aligned}
\|w(s)\|\leq{}\nu\kappa_0\frac{9c_TG^2(\ln{eN})^{1/2}+\frac{|h|}{|f|}G}{N-3\csubL{}G}.
\end{aligned}
\end{equation}
In the following we assume that \eqref{gn_condition1} is satisfied, and
introduce the notation
\begin{equation}\label{notation_bound_w}
 \beta(G,N)=\frac{9c_TG(\ln{eN})^{1/2}+\frac{|h|}{|f|}}{N-3\csubL{}G}\;.
\end{equation}
Thus for the solution $w_n(s)$ of \eqref{NSEwG} with initial condition $w_n(s_0)=0$
\begin{equation}\label{general_bound_n}
\|w_n(s)\|\leq{}\nu\kappa_0{G}\beta(G,N)\;.
\end{equation}

It follows from the general theory of ordinary differential equations that
$w_n$ can be extended to the solution of \eqref{NSEwG} defined on
$[s_0,\infty)$.
% and
%\begin{equation}\label{bound_for_w_-1}
%\begin{aligned}
%&\|w(s)\|\leq{}\kappa_0\nu\alpha^{-1}(36c_TG+\sqrt{2})(\ln{eN})^{-1/2}
%\end{aligned}
%\end{equation}
%if \eqref{bound_for_N} holds.
Therefore the classical proofs of existence for solutions of equation~\eqref{NSE} (see e.g. \cite{CF88} Ch.9 and Lemmas 8.2 and 8.4 in Ch.8) apply almost verbatim to \ref{NSEw}. It follows that \eqref{NSEw} has a unique solution on $[s_0,\infty)$ satisfying $w(s_0)=0$. Now denoting by $w^{(n)}(\cdot)$ the solution of \eqref{NSEw} on $[-n,\infty)$ satisfying $w^{(n)}(-n)=0$, we have
\begin{equation}\label{bound_for_w(n)}
\|w^{(n)}(s)\|\leq \kappa_0\nu{G}\beta(G,N)\text{ for all }s\in [-n,\infty).
\end{equation}
It is also easy to show using \eqref{bound_Aw} that for any fixed, bounded interval
$[s_0,s_1]$ we have
\begin{equation}\label{bound_Aw(k)}
\sup_{n}\int^{s_1}_{s_0}|Aw^{(n)}(s)|^2 ds<\infty.
\end{equation}
It is well known how to prove that there exists a subsequence $\{n_j\}^\infty_{j=1}$ such that $w^{(n_j)}$ converges to a solution $w(s)$ of \eqref{NSEw} in $L^2([s_0,s_1];V)$ satisfying
\begin{equation}\label{bound_w_3}
\|w(s)\|\leq{}\kappa_0\nu{G}\beta(G,N)\text{ for all }s\in[s_0,s_1]\;,
\end{equation}
and moreover, there exists $C=C(G)$, such that
for $s_1-s_0 <1$,
\begin{equation}\label{intAw}
\int_{s_1}^{s_2}|Aw(s)|^2\leq{}C\nu\k0^2\nu{G} \quad \text{for all } s_1, s_2
\text{ such that }s_1-s_0 <1\;,
\end{equation}
(see again \cite{CF88}).

We note that if $w(s)$ is any bounded solution of \eqref{NSEw}, defined on $(-\infty,s_1]$
then in exactly the same way we get the analog of \eqref{bound_Aw} for $w$, from which
by \ref{item3} of Lemma~\ref{useit} we get the bound
\begin{equation}\label{general_bound}
\|w(s)\|\leq{}\nu\kappa_0{G}\beta(G,N)\quad{}\text{ for all }s\in(-\infty,s_1].
\end{equation}

We now prove that this solution, which is bounded in the $\|\cdot\|$ norm, is unique for sufficiently large $N$.
Let $w_1,w_2$ be two bounded solutions of \eqref{NSEw}. Denote $w=w_1-w_2$. Then

\begin{equation*}
\begin{aligned}
\frac{d}{ds}w + \nu Aw +Q_NB(\varphi{}v+w_1,\varphi{}v+w_1)-Q_NB(\varphi{}v+w_2,\varphi{}v+w_2) = 0.
\end{aligned}
\end{equation*}
Taking the scalar product of the above equation with $w$ we get
\begin{equation*}
\begin{aligned}
\frac{1}{2}\frac{d}{ds}|w|^2 &+ \nu \|w\|^2 \\
&= -
(B(\varphi{}v+w_1,\varphi{}v+w_1),w)+(B(\varphi{}v+w_2,\varphi{}v+w_2),w)\\
&= - (B(w,\varphi{}v+w_1),w)- (B(\varphi{}v+w_2,w),w)=-(B(w,\varphi{}v+w_1),w).
\end{aligned}
\end{equation*}
We apply \eqref{Lad}, \eqref{3.0}, and \eqref{bound_w_3} to obtain
\begin{equation*}
|(B(w,\varphi{v}+w_1),w)|\leq{\csubL}|w|\|w\|(\varphi\|v\|+\|w_1\|)\leq{G}(3+\beta(G,N))\csubL\kappa_0\nu|w|\|w\|
\end{equation*}
so that
\begin{equation}
%\begin{aligned}
\frac{1}{2}\frac{d}{ds}|w|^2 + \nu \|w\|^2 \leq G(3+\beta(G,N))\csubL\kappa_0\nu\|w\| |w|.
%\end{aligned}
\end{equation}
Now we use
\begin{equation}\label{3.19.5}
|w|\leq \frac{1}{\kappa_0 N}\|w\|
\end{equation}
 to get
\begin{equation}
%\begin{aligned}
\frac{d}{ds}|w|^2 + 2\kappa_0^2\nu{N^2}(1-\csubL{}G\frac{3+\beta(G,N)}{N}) |w|^2 \leq 0.
%\end{aligned}
\end{equation}

Since $|w|$ is bounded on $(-\infty,\infty)$, we find that $w=0$, when
\begin{equation}\label{gn_condition2}
N>\csubL{}G(3+\beta(G,N))\;.
\end{equation}

We next show that $w(s)$ is continuous in $s$.   Integrating \eqref{NSEw}, we have
$$
|w(s_2)-w(s_1)|=|\int_{s_1}^{s_2} h-Q_NB(\varphi v(s)+w(s))-\nu Aw(s) \ ds |\;.
$$
Using \eqref{agmon}, \eqref{agmon_1}, \eqref{A46b}, together with \eqref{3.0} and \eqref{bound_w_3},
we find that
\begin{align*}
|B(\varphi v)| \le 9\csubA (\nu\k0)^2G^2\;, \quad &|B(\varphi v,w)| \le 3c_B (\nu\k0)^2G^2\beta(G,N)\;, \\
 |B(w,\varphi v)| \le 3c_T (\nu\k0)^2G^2\beta(G,N) \;,
 \quad & |B(w)|  \le {\csubA \k0 \over N^{1/2}} (\nu G\beta(G,N))^{3/2}|Aw|^{1/2}
\end{align*}
It follows that (for $0 < s_2-s_1 < 1$) we have
\begin{align*}
|w(s_2)-w(s_1)|\le &\left\{|h| + \left[(9\csubA+3\beta(G,N)(c_B+c_T)G^2+C(G)\right] (\nu\k0)^2\right\} (s_2-s_1) \\
 & + {\csubA\over N^{1/2}}\k0^{3/2}\nu^{7/4}(G\beta(G,N))^{3/2}[C(G)]^{1/4}(s_2-s_1)^{3/4}\;.
\end{align*}
We work with both the $|\cdot|$- and $\|\cdot \|$-norms to allow for $h$ (hence $f$) to be in $H$, rather
than in $V$.

Note that \eqref{gn_condition2} implies \eqref{gn_condition1}. This concludes the proof of the following lemma.

\begin{lemma}\label{3.1}
Assume that $N$ satisfies \eqref{gn_condition2}, where $\beta(G,N)$
is defined in \eqref{notation_bound_w}. Then the equation \eqref{NSEw} has a unique weak solution defined on the whole real
line, which is bounded in the $\|\cdot\|$ norm. Moreover this weak solution is
in $Y$ with
\begin{equation}
\label{bound_w_X}
\|w(s)\|\leq{}\nu\kappa_0G\beta(G,N)\quad \forall \ s \in \bR \;, \quad \text{and} \quad
|w|_Y\leq{}\nu G{\beta(G,N) \over N}.
\end{equation}
\end{lemma}
This lemma allows us to define the (nonlinear) operator $W$ which takes $v\in X$ to the unique solution of the corresponding equation \eqref{NSEw}. By Lemma~\ref{3.1} we have that
\begin{equation}\label{commute_translation}
W(s;\tau_t(v(\cdot)))=\tau_t W(s;v(\cdot)).
\end{equation}

It is also useful to have a bound for $|Aw|$ if $\|h\|$ is finite. Consider
the following a priori estimates of $Aw$. Taking the scalar
product of \eqref{NSEw} with $A^2w$ we have

\begin{equation}\label{1.1.1.1}
\begin{aligned}
\frac{1}{2}\frac{d}{ds}|Aw|^2 + \nu |A^{3/2}w|^2
+((B(\varphi v+w),A^2w) =(h,A^2w) \le \|h\| |A^{3/2}w|.
\end{aligned}
\end{equation}
Using \eqref{1.8.5}, expanding, then applying  \eqref{flip} (three times) and \eqref{1.7.5},
 we obtain
\begin{equation*}
\begin{aligned}
((B&(\varphi v+w,\varphi v+w),A^2w)\\
&=\varphi^2 (B(v,Av),Aw) +\varphi(B(v,Aw),Aw)
                                                                    + \varphi(B(w,Av),Aw)+(B(w,Aw),Aw)\\
& -\varphi^2 (B(Av,v),Aw) -\varphi(B(Av,w),Aw)
                                                                    - \varphi(B(Aw,v),Aw)-(B(Aw,w),Aw)\\
 &=-\varphi^2 (B(v,Aw),Av)
                                                          - \varphi(B(w,Aw),Av)\\
& -\varphi^2 (B(Aw,v),Av) -\varphi(B(Av,w),Aw)
                                                                    - \varphi(B(Aw,v),Aw)-(B(Aw,w),Aw)\\
\end{aligned}
\end{equation*}
so that by applying for the first term \eqref{titi} and for the rest of terms \eqref{Lad}, followed by \eqref{3.0} and \eqref{bound_w_X}, we have
\begin{equation*}
\begin{aligned}
|((B&(\varphi v+w,\varphi v+w),A^2w)|\\
  \le &\ c_T\varphi^2\|v\| |A^{3/2}w| |Av| (\ln eN)^{1/2} + \csubL \varphi |w|^{1/2} \|w\|^{1/2} |A^{3/2}w| |Av|^{1/2} |A^{3/2}v|^{1/2} \\
 & + \csubL\varphi^2 |Aw|^{1/2}|A^{3/2}w|^{1/2}\|v\| |Av|^{1/2}|A^{3/2}v|^{1/2} \\ & + \csubL \varphi |Av|^{1/2}|A^{3/2}v|^{1/2}\|w\||Aw|^{1/2}|A^{3/2}w|^{1/2} \\
 & + \csubL{}\varphi|Aw||A^{3/2}w|\|v\| + \csubL{} |A^{3/2} w| |Aw| \|w\| \\
 \le & \ c_T(\varphi \|v\|)^2 \k0 N (\ln eN)^{1/2} |A^{3/2}w|
 + \csubL{}\varphi \|v\| { |A^{3/2}w|^2 \over \k0 N}\\
 & + \csubL{}(\varphi \|v\|)^2 \k0 N |A^{3/2}w|
    + \csubL{} \varphi \|v\| {|A^{3/2}w|^2 \over \k0 N} \\
 & + \csubL{} \varphi \|v\| {|A^{3/2}w|^2 \over \k0 N}
 + \csubL{} \k0 \nu \beta(G,N) G {|A^{3/2}w|^2 \over \k0 N} \\
 \le & \csubL{}[9+\beta(G,N)]\nu {G   \over N} |A^{3/2} w|^2
    +9\nu^2\k0^3G^2N(c_T(\ln eN)^{1/2}+\csubL{})|A^{3/2}w| \;.
    \end{aligned}
\end{equation*}
Thus if
\begin{equation}
\label{bound_A_N}
N> \csubL{}G(9+\beta(G,N)) \;,
\end{equation}
we can apply Lemma \ref{useit} (iii) to obtain
\begin{equation}\label{Aw_bound0}
|Aw|\leq\frac{\|h\|+9\nu^2\kappa_0^3G^2N(c_T(\ln{eN})^{1/2}+\csubL{})}
{\nu\k0\left\{N-\csubL{}G(9+\beta(G,N))\right\}}.
\end{equation}

Using the Galerkin approximation we get
the bound \eqref{Aw_bound0} for $w_n$ , provided \eqref{bound_A_N} holds.
Therefore, $w_n\in{}L^\infty(0,T;D(A))$.
From \eqref{1.1.1.1} we obtain that each $w_n$ belongs to a bounded set of $L^2(0,T;D(A^{3/2}))$ for
any $T$.  Therefore as usual choosing a subsequence that converges weakly in
$L^2$ and weakly$^*$ in $L^\infty$, and taking the limit as $n \to \infty$, we obtain the solution of
\eqref{NSEw}. Note that \eqref{bound_A_N} implies \eqref{gn_condition2}. Therefore the following lemma is true,
\begin{lemma}
\label{}
Suppose $h\in D(A^{1/2})$. If $G$ and $N$ satisfy \eqref{bound_A_N}, then the solution to \eqref{NSEw} satisfies
\begin{equation}\label{Aw_bound}
|Aw|\leq\frac{\|h\|+9\nu^2\kappa_0^3G^2N(c_T(\ln{eN})^{1/2}+\csubL{})}
{\nu\k0\left\{N-\csubL{}G(9+\beta(G,N))\right\}}.
\end{equation}
\end{lemma}

We note that the analog of the results in this section with $w \in C_b((-\infty,T],Q_NH)$ instead
of $C_b(\bR,Q_NH)$ in this section remains valid for $v\in C_b((-\infty,T],P_NH)$.

\section{The determining form}\label{detformsec}
\subsection{Lipschitz property of $W$}

Let $v_1,v_2$ be in $X$, and $w_1=W(v_1),w_2=W(v_2)$. Denote $\gamma=w_1-w_2$, $\delta=v_1-v_2$, $\tilde{v}_i=\varphi\left(\frac{\|v_i\|}{\nu\kappa_0}\right)v_i$, $i=1,2$, $\tilde{\delta}=\tilde{v}_1-\tilde{v}_2$.
Subtracting equation \eqref{NSEw} for $w_2$ from that for $w_1$, we obtain
\begin{equation}
\frac{d}{ds}\gamma+\nu A\gamma+Q_NB(\tilde{\delta}+\gamma,\tilde{v}_1+w_1)+Q_NB(\tilde{v}_2+w_2,\tilde{\delta}+\gamma)=0.
\end{equation}
Taking the scalar product with $A\gamma$, we have
\begin{equation}
\begin{aligned}
\frac{1}{2} &\frac{d}{ds}\|\gamma\|^2+\nu |A\gamma|^2 \leq |(B(\tilde{\delta},\tilde{v}_1+w_1),A\gamma)|+|B(\tilde{v}_2+w_2,\tilde{\delta}),A\gamma)|\\ &+|(B(\gamma,\tilde{v}_1),A\gamma)|+|B(\tilde{v}_2,\gamma),A\gamma)|+|(B(\gamma,w_1),A\gamma)|+|B(w_2,\gamma),A\gamma)|.
\end{aligned}
\end{equation}
To estimate the right hand side we apply \eqref{agmon} to the first
term, \eqref{agmon_1} to the second, \eqref{A46a} to the third term, \eqref{flip}, \eqref{1.7.5}, and \eqref{A46c} to the fourth term, and \eqref{A46b}, \eqref{A46a}
respectively to the fifth and sixth term. Using then \eqref{3.0},
\eqref{bound_w_3} and \eqref{3.19.5}, we have
\begin{equation}
\begin{aligned}
\frac{1}{2}\frac{d}{ds}\|\gamma\|^2&+\nu |A\gamma|^2 \leq{}G(3+\beta(G,N))\kappa_0\nu(c_T+c_B)\|\tilde{\delta}\||A\gamma|(\ln{eN})^{1/2}\\
&+3\nu(\csubA+\csubL{})|A\gamma|^{2}\frac{G}{N}+\nu(\csubA+\csubL{})\frac{G\beta(G,N)}{N}|A\gamma|^2,
\end{aligned}
\end{equation}
and thus by Lemma~\ref{useit1} \ref{item21} we have that
\begin{equation}\label{lipsh2.5}
\|\gamma(s)\|\leq\frac{G(3+\beta(G,N))(\ln{eN})^{1/2}(c_T+c_B)}{N-G(3+\beta(G,N))(\csubA+\csubL{})}\sup_{s'\leq{s}}\|\tilde{\delta}(s')\|
\end{equation}
for $G,N$ such that
\begin{equation}\label{gn_condition3}
N>G(3+\beta(G,N))(\csubA+\csubL{}) \;.
\end{equation}

From now on we will consider $G$ and $N$ such that

\begin{equation}\label{gn_condition4}
N>G(9+\beta(G,N))(\csubA+\csubL{})
\end{equation}
which implies \eqref{gn_condition1}, \eqref{gn_condition2}, \eqref{gn_condition3} and \eqref{bound_A_N}.

We derive from \eqref{lipsh2.5} the following two corollaries
\begin{corollary}\label{start_v}
Suppose that $G$ and $N$ satisfy \eqref{gn_condition4}.
Then if $v_1,v_2\in{X}$, $v_1(s)=v_2(s)$ for all
$s\leq{s_0}$ we have that $W(s,v_1)=W(s,v_2)$ for all $s\leq{s_0}$.
\end{corollary}
\begin{proof}
Indeed, if $s\leq{s_0}$, it follows that $\tilde{\delta}(\sigma)=0$ for all
$\sigma\leq{s}$, thus \eqref{lipsh2.5} gives that $\gamma(s)=0$.
\end{proof}
\begin{corollary}
\label{w_lipschitz}
Suppose that $G$ and $N$ satisfy \eqref{gn_condition4}. Then
$W:X\rightarrow{Y}$
 is a Lipschitz function with Lipschitz
constant $L_W$ satisfying
\begin{equation}
\label{w_lipschitz1}
L_W\leq{}4\frac{G(3+\beta(G,N))(\ln{eN})^{1/2}(c_T+c_B)}{N-G(3+\beta(G,N))(\csubA+\csubL{})} \;.
\end{equation}			
\end{corollary}
\begin{proof}
If we have both $\|v_1\| \ge 3\nu\k0G$ and $\|v_2\| \ge 3G$, then $\tilde{\delta}=0$.  If, say $\|v_1\| \le 3G$,
then
\begin{align*}
\|\tilde{v}_1-\tilde{v}_2\| &\le |\varphi\left(\frac{\|v_1\|}{\nu\k0}\right)-\varphi\left(\frac{\|v_2\|}{\nu\k0}\right)|\|v_1\|
 +\varphi\left(\frac{\|v_2\|}{\nu\k0}\right)\|v_1-v_2\| \cr
    & \le \frac{1}{\nu\k0 G} |\|v_1\|-\|v_2\| |3\nu\k0 G + \|v_1-v_2\|  \;
\end{align*}
thus
\begin{equation}
\label{delta_lipschitz}
\|\tilde{\delta}\|\leq{4}\|\delta\|\;.
\end{equation}
\end{proof}

\begin{theorem}
\label{ODE}Suppose $G$ and $N$ satisfy \eqref{gn_condition4}. The equation \eqref{ode_final}
is an ordinary differential equation in the Banach space $C_b(\mathbb{R},P_NH)$.
\end{theorem}
\begin{proof}
It suffices to show that $F(v)=P_Nf-\nu{A}v-P_NB(v+W(v),v+W(v))$ is Lipschitz.
To estimate $L_F$, the Lipschitz constant of $F$, we apply \eqref{titi}, \eqref{3.0},
\eqref{bound_w_3} and Corollary~\ref{w_lipschitz} to obtain for every (nonspecified) $s \in \bR$
\begin{equation*}
\begin{aligned}
\label{}
\eta(s) = |P_NB(\tilde{v}_1&+W(v_1),\tilde{v}_1+W(v_1))-P_NB(\tilde{v}_2+W(v_2),\tilde{v}_2+W(v_2))|\\
 \le {}&|P_NB(\tilde{v}_1+W(v_1),\tilde{v}_1+W(v_1))-P_NB(\tilde{v}_1+W(v_1),\tilde{v}_2+W(v_2))| \\ &+|P_NB(\tilde{v}_1+W(v_1),\tilde{v}_2+W(v_2))-P_NB(\tilde{v}_2+W(v_2),\tilde{v}_2+W(v_2))|\\
 ={}&|P_NB(\tilde{v}_1+W(v_1),\tilde{v}_1-\tilde{v}_2+W(v_1)-W(v_2))|\\
   &+|P_NB(\tilde{v}_1-\tilde{v}_2+W(v_1)-W(v_2),\tilde{v}_2+W(v_2))| \\
  \le {}&c_T(\|\tilde{v}_1+W(v_1)\|+\|\tilde{v}_2+W(v_2)\|)(\ln{eN})^{1/2}\|\tilde{v}_1-\tilde{v}_2+W(v_1)-W(v_2)\|\;.
  \end{aligned}
\end{equation*}
It follows that
$$
 \sup_{s \in \bR} \eta(s) \leq{}2c_T\nu\kappa_0G(3+\beta(G,N))(\ln{eN})^{1/2}(3+L_W)\|v_1-v_2\|_X
$$
and therefore
\begin{equation*}
\begin{aligned}
\|F(v_1)-&F(v_2)\|_X\leq{}\kappa_0N\sup_{s\in \bR}|F(v_1(s))-F(v_2(s))|\\
\leq{}&[\nu\k0^2N^2+2c_T\nu\kappa^2_0NG(3+\beta(G,N))(\ln{eN})^{1/2}(3+L_W)]\|v_1-v_2\|_X
\end{aligned}
\end{equation*}
\end{proof}

We denote the solution operator for \eqref{ode_final} by $S_{\text{DF}}$, i.e.
$S_{\text{DF}}(t)v_0=v(t)$, where $v(t)$ is the solution with initial value $v(0)=v_0$.
Using \eqref{commute_translation} we obtain the following corollary
\begin{corollary}
$\tau_\theta{}S_{\text{DF}}(t)=S_{\text{DF}}(t)\tau_\theta$.
\end{corollary}

It follows that periodicity of trajectories is conserved by  \eqref{ode_final}.

\begin{proposition}
If $v(t)$ is the solution of \eqref{ode_final} and
$v(0)$ is a periodic trajectory with period $T$ then $v(t)$ is periodic with
period $T$ for any $t$.  In particular, if $v(0)$ is constant trajectory, then
$v(t)$ is constant for any $t$.
\end{proposition}

Recall that for $v(t) \in X$ we denote by $v(t,s)$ the value of $v(t)$ at the parameter $s \in \bR$.

\begin{lemma}\label{GivesProj}
Suppose $v$ is a solution of \eqref{ode_final} such that
$\|v(0)\|_{X}\leq{}2\nu\kappa_0G$ and
\begin{equation}\label{vshift}
v(t,s)=v(0,t+s) \quad \text{for all} \quad t,s \in \bR\;.
\end{equation}
Then
$u(t)=v(0,t)+W(t,v(0,\cdot))$ is the solution of \eqref{NSE} satisfying
$$
P_nu(t)=v(0,t)=v(t,0) \quad \text{for all}\quad t \in \bR\;.
$$
\end{lemma}
\begin{proof}
We first note that by \eqref{vshift} and \eqref{commute_translation}
\begin{equation*}\label{wshift}
W(0;v(t,\cdot))=W(0;\tau_t v(0,\cdot))=\tau_tW(0;v(0,\cdot))=W(t;v(0,\cdot))\;,
\end{equation*}
so we can set
\begin{equation}\label{wshift}
w(t)=W(0;v(t))=W(t;v(0))\;.
\end{equation}
Since $\|v(0)\|\leq{}2\nu\kappa_0G$ we have $\varphi=1$, where
$\varphi$ is as in \eqref{varphidef}.  By \eqref{vshift}, \eqref{wshift}, and the definition of the map $W$,
we have
\begin{equation}\label{whalf}
\begin{aligned}
{d\over dt}w(t)+ \nu A w(t)+ Q_NB(v(0,t)+w(t))&= \\
{d\over dt}w(t)+ \nu A w(t)+ Q_NB(v(t,0)+W(t;v(0,\cdot)))&=Q_Nf\;,
\end{aligned}
\end{equation}
as well as
\begin{equation}\label{vhalf}
\begin{aligned}
{d\over dt}v(t)+ \nu A v(t)+ P_NB(v(t)+w(t))&= \\
{d\over dt}v(t,0)+ \nu A v(t,0)+ P_NB(v(t,0)+W(0;v(t,\cdot))&)=P_Nf\;,
\end{aligned}
\end{equation}
Adding the equations \eqref{vhalf} and \eqref{whalf}, we see that
$u(t)=v(0,t)+W(t,v(0,\cdot))$ is a solution of \eqref{NSE}.
\end{proof}
We also have the following property of \eqref{ode_final}
\begin{proposition}
Suppose that $v_1(0,s)=v_2(0,s)$ for $s\leq{s_0}$.  Then $v_1(t,s)=v_2(t,s)$ for
$s\leq{s_0}$ for any $t$.
\end{proposition}
\begin{proof}
By Corollary~\ref{start_v} the map $W$ is well-defined on the space
$C_b((-\infty,s_0],P_NH)$.  Thus we can consider the differential equation
\eqref{ode_final} on this space.  It is still an ordinary differential equation, and
so through each point there is only one solution.  Each solution of
\eqref{ode_final} on the whole $X$ gives a solution of \eqref{ode_final} on
$C_b((-\infty,s_0],P_NH)$, by restriction.  Since solutions $v_1(t)$ and $v_2(t)$
of the restriction of \eqref{ode_final} on $C_b((-\infty,s_0],P_NH)$ coincide for
$t=0$, they coincide for all $t$.
\end{proof}
\section{Absorbing ball}\label{absballsec}
We will now prove that the determining form \eqref{ode_final} has an absorbing ball in $X$.
Recall that this means an estimate of the type \eqref{AbsTime}, holds
but with $\|\cdot\|$ replaced by $|\cdot|_X$.
Let $g=P_Nf$ and take the
scalar product with $v=v(t,s)$ (and denote $W(v)$ by $w$) to obtain
\begin{equation*}
%\label{dissip_v}
\begin{aligned}
\frac{1}{2}\frac{d}{d{t}}&|v|^2+\nu\|v\|^2\leq|g||v|+|\varphi(B(v,v),w)|+|\varphi(B(w,v),w)|\\
\leq{}&|g|\|v\|+c_B\varphi\|v\|^2|w|(\ln{eN})^{1/2}+\csubL{}\varphi|w|\|w\|\|v\|\\
\leq{}&\kappa_0\nu^2\frac{|g|}{|f|}G\|v\|+\nu{}c_B\|v\|^2\frac{G\beta(G,N)}{N}(\ln{eN})^{1/2}+\kappa_0\nu^2\csubL{}\frac{G^2\beta(G,N)^2}{N}\|v\|,
\end{aligned}
\end{equation*}
where we used inequalities \eqref{agmon},\eqref{Lad}, \eqref{bound_w_3} and \eqref{3.19.5}.
It follows from Lemma~\ref{useit} that if
\begin{equation}
\label{4.11.5}
N>c_BG\beta(G,N)(\ln{eN})^{1/2}
\end{equation}
holds, we have the following
\begin{proposition}
\label{absorbing_ball}
If \eqref{4.11.5} and \eqref{gn_condition4} hold, then
\begin{equation}
\label{abs_radius}
|v| \le\nu{G}\frac{\frac{|g|}{|f|}+\csubL{}G\frac{\beta(G,N)^2}{N}}{1-c_B\frac{G\beta(G,N)}{N}(\ln{eN})^{1/2}}\;.
\end{equation}
\end{proposition}
While the relation $\|v\|\leq{\kappa_0N}|v|$ provides an immediate bound on
$\|v\|$ provided \eqref{4.11.5} holds, we can derive a sharper bound on
$\|v\|$.  As above, taking the scalar
product with $Av$ we have
\begin{equation}
\label{dissip_v2}
\begin{aligned}
\frac{1}{2}\frac{d}{d{t}}\|v\|^2+\nu|Av|^2&=(g,Av)-(B(\tilde{v}+w,\tilde{v}+w),Av)\\
&\leq{} |g||Av|-(B(\tilde{v},w),Av)-(B(w,\tilde{v}),Av)-(B(w,w),Av)\\
&\leq{}|g||Av|+c_T(2\|\tilde{v}\|\|w\||Av|+\|w\|^2|Av|)(\ln{eN})^{1/2}\\
&\leq{}\nu^2\kappa_0^2\left[\frac{|g|}{|f|}G+c_T(\ln{eN})^{1/2}G^2(6+\beta(G,N))\beta(G,N)\right]|Av|\;,
\end{aligned}
\end{equation}
where $v=v(t,s)$ and \eqref{dissip_v2} is valid at each $s\in \bR$.
Now applying Lemma~\ref{useit} (i) and then (ii), we obtain
\begin{theorem}
\label{absorbing_ball1}
%The ODE~\eqref{ode_final} is dissipative.
If \eqref{gn_condition4} holds, then the ball centered at the origin, with
radius (strictly) greater than
\begin{equation}
\label{rad_2}
r_0=\nu\kappa_0G\left[\frac{|g|}{|f|}+c_T(\ln{eN})^{1/2}G(6+\beta(G,N))\beta(G,N)\right]
\end{equation}
is an absorbing ball in $X$ for \eqref{ode_final}.
\end{theorem}

\begin{remark} \label{smallerball}
The $O(G^2)$ estimate (up to a logarithm) for the radius of the absorbing ball for \eqref{ode_final} in \eqref{rad_2}  is significantly larger than the  $O(G)$ estimate for radius of the
absorbing ball for the NSE.  These estimates can be brought closer if we assume $N$ is larger.
Indeed, if we take for some $\gamma > 1$
\begin{equation} \label{bigNcond}
N \ge (6 c_T G (\ln e N)^{1/2})^{\gamma}\;,
\end{equation}
then
$$
\beta(G,N) \le { 9 N^{{1 /\gamma}}/(6c_T) + 1 \over N- N^{{1 \over \gamma}}/2}
\le {3N^{{1 \over \gamma}} + 2 \over N} \le 5 N ^{{1-\gamma \over \gamma}}\;,$$
and consequently
$$
G (\ln e N)^{1/2}\beta(G,N) \le {5 \over 6c_T}N^{{2-\gamma\over \gamma}}\;.
$$
Thus, if equality holds in \eqref{bigNcond} with $\gamma=3/2$, we have (up to a logarithm) that $r_0=O(G^{3/2})$, and
if \eqref{bigNcond} holds with
$\gamma=2$, we have $r_0=O(G)$.
\end{remark}

Let $\calB_r$ be an absorbing ball  in $X$, centered at $0$ with radius $r>r_0$, and let $\mathcal M=\cap_{t>0}S_{\text{DF}}(t)\calB_r$.  Note that $\mathcal M$ is
independent of $r$: $v_0\in{X}$ belongs to $\mathcal{M}$ if and only if there is
a bounded solution $v:(-\infty,\infty)\rightarrow{X}$ of \eqref{ODE} defined
on the whole real line such that $v(0)=v_0$.   We can also characterize $\mathcal M$
as the maximal bounded set that is invariant (i.e.  $S_{\text{DF}} (t)\cM=\cM$ $ \forall t \ge 0$), both forward and backward in time, under
the flow generated by  \eqref{ode_final}.    We do not know if  $\cM$ is the analog for the determining
form  \eqref{ode_final} of the global attractor of the Navier-Stokes equations (see discussion leading to
\eqref{attractordef}).

By Theorem~\ref{absorbing_ball1} $\mathcal{M}$ is a subset of $\calB_{r_0}$.

\begin{lemma}\label{bound}
If $v$ is a bounded solution of \eqref{ODE} defined on the whole real line and $\|W(v(t))\|_Y\le\nu\k0\sqrt{2G/c_L}$ for all $t$ then $\|v(t)\|_X\le 3\nu\k0 G$ for all $t$. In particular this is true if $N\geq{}\sqrt{2\csubL}(9c_TG^{3/2}(\ln{eN})^{1/2}+G^{1/2})$ and $N\geq 6\csubL G$.
\end{lemma}
\begin{proof}
Let $w(t)=W(v(t))$ and $w(t,s)=w(t)(s)$. Suppose that $\|v(t',s')\|> 3\nu\k0 G$ for some $t'$ and $s'$. Then since $v$ is continuous there are $t_0<t'<t_1$ suth that in the interval $[t_0,t_1]$ we have $\|v(t,s')\|\geq 3\nu\k0 G$. On this interval the map $t\mapsto v(t,s')$ satisfies the equation
\[
\frac{d}{dt}v(t,s')+\nu Av(t,s')+B(w(t,s'))=g.
\]
Taking its scalar product with $Av(t,s')$ we obtain inequality
\[
\frac{d}{dt}\|v(t,s')\|^2+\nu|Av(t,s')|^2\leq|g||Av(t,s')|+|(B(w(t,s'),Av(t,s')),w(t,s'))|
\]
and using \eqref{Lad} we have that $|(B(w,Av),w)|\le\csubL|w|\|w\||A^{3/2}v|\le\csubL\|w\|_Y^2|Av|$, and hence
\[
\frac{d}{dt}\|v(t,s')\|^2+\nu|Av(t,s')|^2\leq(G\nu^2\k0^2+\csubL\|w\|^2_Y)|Av(t,s')|.
\]
Note that by Lemma~\ref{useit} (i) $t_0=-\infty$, since $G\nu\k0+\csubL\|w(t)\|^2/(\nu\k0)<3G\nu\k0$, but by the Lemma~\ref{useit} (iii) this is impossible.
\end{proof}
%, and nonempty by Theorem
\begin{remark}
It is easy to give an analytic characterization of the solutions of \eqref{ode_final} in $\cM$. Indeed we have
\begin{equation*}
v(t)=e^{-\nu(t-t_0)A}v(t_0)+\int^t_{t_0}e^{-\nu(t-\theta)A}[g-P_NB(\tilde{v}(\theta)+W(v(\theta)))]d\theta,
\end{equation*}
and letting $t_0\rightarrow-\infty$ we easily obtain the relation
\begin{equation}\label{star}
v(t)=\int^t_{-\infty}e^{-\nu(t-\theta)A}[g-P_NB(\tilde{v}(\theta)+W(v(\theta)))]d\theta\quad{t>0},
\end{equation}
in $X$. Clearly \eqref{star} implies that $v(t)$ is in $\cM$.

In the case when
\[
v(t,s)=P_Nu(t+s),\quad{s,t\in\mathbb{R}},
\]
where $u(t)$ is a solution of \eqref{NSE} in $\cA$, it is easy to see that the relation \eqref{star} is equivalent to
\begin{equation}\label{doublestar}
v(s)=\int^s_{-\infty}e^{-\nu(s-\sigma)A}[g-P_NB(\tilde{v}(\sigma)+W(v(\sigma)))]d\sigma\quad{s\in\mathbb{R}}.
\end{equation}
It follows that
\[
v(s)=P_Nu(s),\quad{s\in\mathbb{R}},
\]
where $u(s)$ is a solution of \eqref{NSE} in $\cA$ if and only if $v\in{}X$ and satisfies (in $H$) the relation \eqref{doublestar}.
\end{remark}

\section{Traveling waves}\label{travwavesec}
We say that $v(t,s)=S_{\text{DF}}(t)v_0(s)$ is a traveling wave solution for \eqref{ode_final}  if
\begin{equation}\label{eq_t1}
v(t,s)=v_0(t+s).
\end{equation}
Denote $w(t,s)=W(s;v(t,\cdot))$,
${u}(t,s)=v(t,s)+w(t,s)$, and define $w_0\in{Y}$,
${u}_0\in{}C_b(\mathbb{R},H)$ by $w_0(s)=w(0,s)$, ${u}_0(s)={u}(0,s)$.
Suppose $v$ is a traveling wave solution.  Then by
\eqref{eq_t1}, \eqref{commute_translation}
$w(t,s)=w_0(t+s)$, thus $u(t,s)=u_0(t+s)$.  We
will provide an equation for $u_0$ and use it to derive the
properties of translation invariant solutions.
Using \eqref{eq_t1}, equation \eqref{ode_final}
can be written as
\begin{equation}\label{de_v}
\frac{d}{dt}v_0+\nu{}Av_0=P_Nf-P_NB(\varphi{}v_0+w_0),
\end{equation}
while \eqref{NSEw} is written as
\begin{equation}\label{de_w}
\frac{d}{ds}w_0+\nu{}Aw_0=Q_Nf-Q_NB(\varphi{}v_0+w_0).
\end{equation}
Changing $t$ to $s$ in \eqref{de_v} and combining the result with \eqref{de_w}, we have the equation for $u_0$
\begin{equation}\label{de_u}
\frac{d}{ds}u_0+\nu{}Au_0=f-B(\varphi{}v_0+w_0).
\end{equation}
Taking the scalar product with $Au_0$, we have
\begin{equation}\label{energy_u}
\frac{1}{2}\frac{d}{ds}\|u_0\|^2+\nu|Au_0|^2=(f,Au_0)-(B(\varphi{}v_0+w_0),Au_0).
\end{equation}
Using \eqref{moveu} and \eqref{ortho}, we obtain
\begin{equation*}
\begin{aligned}
(B(\varphi{}v_0+w_0),Au_0)
&=\varphi[(B(v_0,w_0),Av_0)+(B(w_0,v_0),Av_0)]+(B(w_0,w_0),Av_0)\\
+&\varphi^2(B(v_0,v_0),Aw_0)+\varphi[(B(v_0,w_0),Aw_0)+(B(w_0,v_0),Aw_0)]\\
&=(\varphi^2-\varphi)(B(v_0,v_0),Aw_0)+(1-\varphi)(B(w_0,w_0),Av_0)\;,
\end{aligned}
\end{equation*}
so by applying \eqref{moveu} again, we have
\begin{equation*}\label{BAllEndAll}
\begin{aligned}
|(B(\varphi{}v_0&+w_0),Au_0)|\le \varphi|(B(v_0,v_0),Aw_0)|+|(B(w_0,w_0),Av_0)|\\
 &\le \varphi|(B(v_0,w_0),Av_0)|+\varphi|(B(w_0,v_0),Av_0)|+|(B(w_0,w_0),Av_0)|\;.
\end{aligned}
\end{equation*}
We now apply \eqref{titi}, \eqref{3.0}, \eqref{bound_w_X} to obtain
\begin{equation*}
%\label{ineq_t2}
\begin{aligned}
\frac{1}{2}\frac{d}{ds}\|u_0\|^2+\nu|Au_0|^2 &\leq{}
|f||Au_0|+c_T(6+\beta(G,N))\nu^2\kappa_0^2G^2\beta(G,N)(\ln{eN})^{1/2}|Av_0|\\
&\leq{}
\nu^2\kappa_0^2G[ 1+c_T(6+\beta(G,N))G\beta(G,N)(\ln{eN})^{1/2}]|Au_0|\end{aligned}
\end{equation*}
and consequently by Lemma~\ref{useit} (iii)
\begin{equation}\label{5.13}
\|u_0(s)\|\leq{}\nu\kappa_0G[ 1+c_T(6+\beta(G,N))G\beta(G,N)(\ln{eN})^{1/2}] \;.
\end{equation}
We next calculate an absolute constant $c'_T$ such that if
\begin{equation}\label{constcond}
N \ge c'_T G\beta(G,N) \ln{eN}\;,
\end{equation}
then
$$
c_T(6+\beta(G,N))G\beta(G,N)(\ln{eN})^{1/2} \le 1\;,
$$
so that \eqref{5.13} implies
\begin{equation}\label{GivesProj2}
\|v_0(s)\|\leq{}\|u_0(s)\|\leq{}2\nu\kappa_0G\;.
\end{equation}
First note that for $c'_T \ge 18c_T$, we have $\beta(G,N) \le 2$, for all $N \ge 1$.
Then estimate
\begin{align*}
8c_TG\beta(G,N)(\ln{eN})^{1/2} &= 8c_T\frac{9c_TG \ln{eN} + ( \ln{eN})^{1/2}|h|/|f|}{N-3c_TG(\ln{eN})^{1/2}} \\
& \le 8c_T\frac{9 c_T N/c'_T+N/c'_T}{N-3c_TN/c'_T}= \frac{72 c_T^2 + 8c_T}{c'_T-3c_T}\;,
\end{align*}
so that it suffices in \eqref{constcond} to take
\begin{equation}\label{cvalue}
c'_T\ge72c_T^2+11c_T\;.
\end{equation}

\begin{theorem}
\label{inv}
 Suppose \eqref{constcond} holds for $c'_T$ as in \eqref{cvalue}.  Then a solution v(t) of \eqref{ode_final} is traveling wave solution if and only if v(0) is a projection of a solution of \eqref{NSE} in $\mathcal {A}$.
\end{theorem}
\begin{proof}
The fact that any traveling wave solution of \eqref{ode_final} is the
projection of a solution of \eqref{NSE} follows from \eqref{GivesProj2} and Lemma \ref{GivesProj}.  To prove the
converse, recall that any solution $u(t)$ of
the \eqref{NSE} which lies on the attractor can be extended to a solution to the complexified
NSE which is analytic and
bounded on some strip $|\Im t|<c$ function into $H$ (see e.g. \cite{CF88}). Let $u(t)$ be a solution of \eqref{NSE} defined and
bounded on the whole real line.  Consider a curve in $X$,
$t\mapsto{}(s\mapsto{\mp(t+s)})$, where $\mp=P_Nu$. We need to show that this curve is
differentiable and its derivative is $s\mapsto{d\mp/dt(t+s)}$; thus we need to prove that for any real $t$
\begin{equation}\label{needthis}
\sup_{s\in\mathbb{R}}\|\frac{\mp(t+\Delta{t}+s)-\mp(t+s)}{\Delta{t}}-\frac{d}{dt}\mp(t+s)\|\rightarrow{0}
\quad \text{as} \quad \Delta{t} \to 0\;.
\end{equation}
 Since $\mp$
is analytic and bounded in the strip $|\Im\zeta|<c$, its second derivative is
also bounded in a smaller strip.    The limit in \eqref{needthis} now follows from
Taylor's theorem, specifically
\begin{equation*}
\sup_{s\in\mathbb{R}}\|\frac{\mp(t+\Delta{t}+s)-\mp(t+s)}{\Delta{t}}-\frac{d}{dt}\mp(t+s)\|
\le \sup_{s \in \bR} \Delta t \sup_{\tau \in [t,t+t+\Delta t]}\|\frac{d^2}{dt^2} \mp(\tau+s)\|\;.
\end{equation*}

\end{proof}

\begin{remark}\label{rem}
Note that  a function $v(t,s-t)$ being independent of $t$ is
equivalent to $v(t,s)=v(0,s+t)$, for all $s,t$; thus if $\|v(0)\|_X \le 2\nu\k0 G$, then
$v(t)$ is the $P_N$-projection of a solution on the global attractor  of the Navier-Stokes equations.
\end{remark}

\section{Equicontinuity of $W$}\label{equisec}
%The following results may be useful in the study of the asymptotic behavior
%of the solutions of \eqref{ode_final}.

We have from \eqref{bound_Aw} and \eqref{general_bound} that
\begin{equation}
\int^{s_1}_{s_0}|Aw|^2\leq{}E_1(s_1-s_0)+E_0,
\end{equation}
for some constants $E_0,E_1$.  We also have for any $u \in H$
\begin{equation*}
\begin{aligned}
%&|B(\varphi{v}+w,\varphi{v}+w)|\leq{}\sup_{h'\in H, |h'|\leq{1}}
 |(B(\varphi{v}+w),u)| \leq{}9c_BG^2\nu^2\kappa_0^2(\ln{eN})^{1/2}|u| &+ 3(\csubA+c_T) \nu^2\kappa_0^2G^2\beta(G,N) (\ln eN)^{1/2}|u| \\
 &+\csubA\nu\kappa_0{G\over N}\beta(G,N) |Aw||u| \;;
\end{aligned}
\end{equation*}
thus
\begin{equation}
|B(\varphi{v}+w)|\leq{}E_2+E_3|Aw|
\end{equation}
where
\begin{equation*}
\begin{aligned}
E_2=&[9c_B+ 3(c_B+c_T) \beta(G,N) ]\nu^2\k0^2G^2 \beta(G,N) (\ln eN)^{1/2}\\
E_3=&\csubA\nu\kappa_0{G\over N}\beta(G,N) \;.
\end{aligned}
\end{equation*}
Integrating \eqref{NSEw} from $s_0$ to $s_1$ we have
\begin{equation*}
\begin{aligned}
|w(s_1)-w(s_0)|\leq{}&\nu\int^{s_1}_{s_0}|Aw|+\int^{s_1}_{s_0}|B(\varphi{}v+w,\varphi{}v+w)|+\int_{s_0}^{s_1} |h| \\
\leq{}&(\nu+E_3)(s_1-s_0)^{1/2}[E_1(s_1-s_0)+E_0]^{1/2}+(E_2+|h|)(s_1-s_0)
\end{aligned}
\end{equation*}

We have proved the following
\begin{lemma}
The family $\{W(v)(\cdot)|v\in{C_b(\mathbb{R},P_NH)}\}$ is equicontinuous on
$\mathbb{R}$.
\end{lemma}

Let $v(t)$ be a solution of \eqref{ode_final}, and denote $w(s,t)=W(v(t))(s)$.  Then
by Lemma~\ref{w_lipschitz}
\[
\|w(s,t)-w(s,t_0)\|\leq{}L_W\|v(t)-v(t_0)\|_X.
\]
 Since $v(t)$ is a solution of an ODE, it is continuous.  Thus for any finite
interval $[-T,T]$ there is a constant $C_T$, such that
\[
\|v(t)-v(t_0)\|_X\leq{}C_T|t-t_0|,\quad{}t,t_0\in{}[-T,T].
\]
This proves
\begin{proposition}\label{equi}
For any solution $v(t)$ of the equation~\eqref{ode_final} the function
$w(s,t)=W(v(t))(s)$ is equicontinuous on $\mathbb{R}^2$. Moreover, there are constants $c,c_1$ and $C_T$  such that
\[
|w(s,t)-w(s_0,t_0)|\leq{}c|s-s_0|^{1/2}+c_1|s-s_0|+\frac{C_T}{N}|t-t_0|,
\]
whenever $t,t_0\in[-T,T]$.
\end{proposition}

We now assume that \eqref{gn_condition4}
holds, and that $h \in D(A^{1/2})$.  Then by \eqref{Aw_bound} we have
\begin{equation}\label{AWCNbound}
|AW(v)(s)| \le \Gamma;, \quad \text{for all} \ v \in X, \ s \in \bR\;,
\end{equation}
where $\Gamma$ is the right handside of \eqref{Aw_bound}.  We have therefore that
\begin{equation}\ \label{enhancedequi}
 \begin{aligned}
\|&w(s,t)-w(s_0,t_0)\| \le (2\Gamma)^{1/2}|w(s,t)-w(s_0,t_0)|^{1/2} \\
  &\le  (2\Gamma)^{1/2}
 \left[c|s-s_0|^{1/2}+c_1|s-s_0|+\frac{C_T}{N}|t-t_0| \right]^{1/2}
 \end{aligned}
\end{equation}
for all $s_0,s_1 \in \bR$, and all $t,t_0 \in [-T,T]$.

%Moreover, since $\{v: |Av| \le ({\text{const.}})\nu\k0^2\}$ is
% compact in $V$, the set $\{w_n(t_n,s_n)=W(v_n(t_n))(s_n): n=1,2,\ldots\}$
%is relatively compact in $V$ for any $\{s_n\}_{n=1}^{\infty}, \{t_n\}_{n=1}^{\infty} \subset \bR$,
%and $\{v_n(\cdot)\}_{n=1}^{\infty}$ a sequence of solutions to the determining form \eqref{ode_final}.

Recall the extension of the Arzela-Ascoli theorem to the case of Banach-valued functions.
\begin{theorem}\label{AAext0}
Let $K$ be a compact metric space, $\cX$ a Banach space, $\Phi_n:k \mapsto \Phi_n(k) \in \cX$,
with $\{\Phi_n\}_{n=1}^{\infty}$ equicontinuous and  $\{\Phi_n(k)\}_{n=1}^{\infty}$
relatively compact for each $k \in K$.   Then there exists a subsequence $\{\Phi_{n_j}\}_{j=1}^{\infty}$
which converges in $C(K;\cX)$.
\end{theorem}

\section{Weak limit of time translated solutions} \label{weaksec}

According to Remark \ref{rem} (Section 7), a solution $v$ to the determining form
\eqref{ode_final} in the set $\mathcal{M}$ is the $P_N$-projection of a solution of the Navier-Stokes equations \eqref{NSE} in
its global attractor $\cA$ if and only if
\begin{equation} \label{bndforatt}
\|v(t,s)\| \le 2\nu\k0^2 G\;, \quad \forall \ t,s \in \bR\;,
\end{equation}
and
$$
v(t+t_n,s-(t+t_n))=v(0,s)\;, \quad \forall \ t,s \in \bR\;,
$$
for any sequence $\{t_n\}$.  It is therefore instructive to study the asymptotic behavior
of a sequence
\begin{equation}\label{vndef}
v_n(t,s)=v(t+t_n,s-(t+t_n))\;, \quad n=1,2,\ldots
\end{equation}
where $t_n \to \infty$, and $v(\cdot)\in \cM$ is a solution of \eqref{ode_final}
satisfying \eqref{bndforatt}.   We have the following lemma,
%the proof of which is similar to the proof of Lemma \ref{convlem}.

\begin{lemma}\label{definitionbar}

There is a subsequence $\{n_j\}$ and functions $\bar v$, $\bar w$ and $\bar R$ such that $\bar v,\bar R\in C(\mathbb{R},L^\infty(\mathbb{R},H)_{\text{weak}^*})$, $\bar w\in C(\mathbb{R}^2,V)$ and the following hold for each $T,S$
\begin{equation}\label{wlim1}
\lim_j\sup_{t\in[-T,T]}|(v_{n_j}(t,\cdot)-\bar v(t,\cdot),\xi)|=0 \quad\text{ for }\xi\in L^1(\mathbb{R},H)
\end{equation}
\begin{equation}\label{wlim2}
\lim_j\sup_{t\in[-T,T]}|(B(v_{n_j}(t,\cdot)-\bar v(t,\cdot))+\bar R(t,\cdot),\xi)|=0 \quad\text{ for }\xi\in L^1(\mathbb{R},H)
\end{equation}
\begin{equation}\label{wlim3}
w_{n_j}(t,s) =W(\tau_{t_{n_j}}v)(s-(t+t_{n_j})) \to \bar w(t,s) \quad (t,s \in \bR\;)
\end{equation}
in $C([-T,T]\times[-S,S];V)$, for all $T,S > 0$.
\end{lemma}
\begin{proof}
It follows from \eqref{ode_final} and \eqref{AWCNbound}
that  there exists a constant $C$
$$
|\frac{\partial v_n(t,s)}{\partial t} | \le C \quad (\forall \ t,s,n )
$$

Choose a dense subset $\{\xi_k\}$ in the unit ball of $L^1(\mathbb{R},H)$. Then it is well-known and easy to check that
\[
d(p,p')=\sum_k {|(p-p',\xi_k)| \over 2^k(1+|(p-p',\xi_k)|)}
\]
defines a distance in the ball of any fixed size in $L^\infty(\mathbb{R},H)$. Let us compute $d(v_n(t,\cdot),v_n(t',\cdot))$.
Note first that
\begin{equation*}
\begin{aligned}
|&\int_\mathbb{R}(v_n(t,s)-v_n(t',s),\xi_k(s))\ ds|\\
&= |\int_\mathbb{R}(v(t+t_n,s-t-t_n)-v(t'+t_n,s-t'-t_n),\xi_k(s))\ ds| \\
&\le
C|t-t'|+|\int_\mathbb{R}(v(t+t_n,s-t-t_n)-v(t+t_n,s-t'-t_n),\xi_k(s))\ ds|\\
&=C|t-t'|+|\int_\mathbb{R}(v(t+t_n,s-t_n),\xi_k(s+t)-\xi_k(s+t'))| \ ds\\
&\le  C|t-t'|+3G\nu\k0 \|\xi_k-\tau_{t-t'}\xi_k\|_{L^1}
\end{aligned}
\end{equation*}
which, provided $k$ is fixed, goes to $0$ when $t-t'$ goes to $0$, uniformly
in $n$. It follows that for any fixed $\varepsilon$ there is a $\delta$ such
that $|t-t'|<\delta$ implies that $d(v_n(t,\cdot),v_n(t',\cdot))<\varepsilon$
for all $n$.  This means that the set of functions $t\mapsto v_n(t,\cdot)$ is an
equicontinuous family of functions in
$C([-T,T],L^\infty(\mathbb{R},H)_{\text{weak}^*})$
for any $T$. Note that $|v_n(t,s)|\le \|v\|_X/\k0$, therefore $v_n(t,\cdot)$ belongs to
the ball of radius $\|v\|_X/\k0$ in $L^\infty(\mathbb{R},H)$, which by the
Banach-Alaoglu Theorem is compact
in $L^\infty(\mathbb{R},H)_{\text{weak}^*}$. Using the Arzela-Ascoli Theorem
\ref{AAext0} and the diagonal process, we obtain a function $\bar v$ in
$C(\mathbb{R},L^\infty(\mathbb{R},H)_{\text{weak}^*})$ such that for some
subsequence of $\{v_{n_j}\}$, \eqref{wlim1} holds for any $T$. Without loss of
generality we assume that \eqref{wlim1} holds for the sequence $\{v_{n}\}$.

To deal with the sequence $R_n(t,\cdot)=-B(v_n(t,\cdot)-\bar v(t,\cdot))$ note
that it is a composition of an equicontinuous family $\{v_n(t,\cdot)-\bar
v(t,\cdot)\}$ in $C(\mathbb{R},L^\infty(\mathbb{R},P_NH)_{\text{weak}^*})$ and
a map $L^\infty(\mathbb{R},P_NH)_{\text{weak}^*}\to L^\infty(\mathbb{R},H)_{\text{weak}^*}$
induced by $B$.  Since $P_NH$ is finite dimensional, $B$ is a continuous
quadratic form on $P_NH$, and hence the induced map is continuous. It follows
that the family $R_n(t,\cdot)\in
C(\mathbb{R},L^\infty(\mathbb{R},H)_{\text{weak}^*})$
is also equicontinuous. Now note that by \eqref{A46b}
\begin{equation*}
\begin{aligned}
|B(v_n(t,s)-\bar v(t,s))|&
        \leq \csubA |v_n(t,s)-\bar v(s)|^{1/2}|Av_n(t,s)-A\bar
        v(s)|^{1/2}\|v_n(t,s)-\bar v(s)\| \\
& \leq\csubA(\k0 N)^2|v_n(t,s)-\bar v(s)|^2\le 2\csubA(\k0 N)^2\|v\|_X,
\end{aligned}
\end{equation*}
thus $R_n(t,\cdot)$ is in the ball of radius $2\csubA(\k0 N)^2\|v\|_X$ in $L^\infty(\mathbb{R},H)$. Therefore, again by Theorem \ref{AAext0} and the diagonal procedure, there exists $\bar R(t,\cdot)\in C(\mathbb{R},L^\infty(\mathbb{R},H)_{\text{weak}^*})$ and a subsequence such that \eqref{wlim2} holds for any $T$.  Once again, without loss of generality, we assume that \eqref{wlim1} and \eqref{wlim2} hold for the whole sequence.

Finally, $w_n$ is equicontinuous in $C(\mathbb{R}^2,V)$ by Proposition \ref{equi}, and for each $t,s$ $w_n(t,s)$ is precompact in $V$ by \eqref{AWCNbound}, thus, by Theorem \ref{AAext0}, \eqref{wlim3} holds for some subsequence $v_{n_j}$.
\end{proof}

The main result in this section is the following.

\begin{proposition} \label{convprop}
Suppose that $N$ and $G$ satisfy \eqref{gn_condition4}.
Let $\bar v$, $\bar w$ and $\bar R$ be as in Lemma~\ref{definitionbar}.
If $\bar v(t,s)$ and $\bar w(t,s)$ are both independent of $t$, then $\bar R$ is also independent of $t$ and
$$
\bar u(s)=\bar v(s) + \bar w(s)
$$
satisfies the Navier-Stokes equation
\begin{equation}\label{wstress}
\frac{d}{ds}\bar u(s)+\nu{A}\bar u(s)+B\left(\bar u(s),\bar u(s))\right)=f+\bar R(s)\;, \quad s \in \bR\;.
\end{equation}
\end{proposition}

The function $\bar R$ can be viewed as the analog of the Reynolds stress due to averaging
the Navier-Stokes equations in the theory of turbulence \cite{FMRT,FJMR}
Thus, $\bar u(\cdot)$ is a solution
to the Navier-Stokes equations \eqref{NSE} {\it with the original force $f$},
if and only if $\bar v(t,s)$ and
$\bar w(t,s)$ are independent of $t$, and $\bar R(s)=0$.  We do not know if any of these conditions
are superfluous.

\begin{proof}
Using a change of variables, we write
\begin{equation}\label{wnstart}
\begin{aligned}
w_n(t,s)&=w (t+t_n,s-(t+t_n))=W(\tau_{t_n}v)(s-(t+t_n)) \\
             &= \int_{-\infty}^{s-(t-t_n)} e^{-\nu(s-t-t_n-\sigma)A}
                   \left[h-Q_NB(v(t+t_n,\sigma)+w(t+t_n,\sigma))\right] \ d \sigma \\
             &= (\nu A)^{-1}h-\int_{-\infty}^{s} e^{-\nu(s-\sigma)A}
                   \left[Q_NB(v_n(t,\sigma)+w_n(t,\sigma))\right] \ d \sigma \;.
\end{aligned}
\end{equation}
Now write
\begin{equation}\label{wintform}
\begin{aligned}
&\bar w(s) -(\nu A)^{-1}h + \int_{-\infty}^s e^{-\nu(s-\sigma)A} Q_N [B(\bar v(\sigma)+\bar w(\sigma))-\bar R(t,\sigma)] d \sigma\\
&= \bar w(s) - w_n(t,s) - \\ & \int_{-\infty}^{s} e^{-\nu(s-\sigma)A} Q_N \left[B(v_n(t,\sigma)+w_n(t,\sigma))-B(\bar v(\sigma)+\bar w(\sigma))-B(v_n(t,\sigma)-\bar v(\sigma)) \right]\ d \sigma \\
& -\int_{-\infty}^s e^{-\nu(s-\sigma)A} Q_N \left[B(v_n(t,\sigma)-\bar v(\sigma))+\bar R(t,\sigma)\right] \ d \sigma \;.
\end{aligned}
\end{equation}
Notice that
\begin{equation}\label{Bexp}
\begin{aligned}
& B(v_n+w_n)-B(\bar v+\bar w)-B(v_n-\bar v)= \\ & B_{\text{sym}}(v_n+\bar w,w_n-\bar w)+B(w_n-\bar w)+B_{\text{sym}}(\bar v+\bar w,v_n-\bar v)
\end{aligned}
\end{equation}
where
$$
B_{\text{sym}}(u_1,u_2)=B(u_1,u_2)+B(u_2,u_1)\;.
$$
Using \eqref{A46b}, \eqref{AWCNbound} and \eqref{A46a} as well, we find that
\begin{align*}
|B(v_n+\bar w,w_n-\bar w)| &\le \csubA |v_n+\bar w|^{1/2} |A(v_n+\bar w)|^{1/2}\|w_n - \bar w\| \\
&\le \frac{\csubA}{\k0}\left(|Av_n|+|A\bar w|\right)\|w_n - \bar w\|\\
&\le \csubA \left( N\|v\|_X+\frac{\Gamma}{\k0}\right) \|w_n - \bar w\|\;.
\end{align*}
\begin{align*}
|B(w_n-\bar w,v_n+\bar w)| &\le \csubL |w_n - \bar w|^{1/2}\|w_n - \bar w\|^{1/2}\|v_n+\bar w\|^{1/2} |A(v_n+\bar w)|^{1/2} \\
&\le \frac{\csubL}{\k0 N^{1/2}}\left(|Av_n|^2+|A\bar w|^2\right)^{1/2} \|w_n- \bar w\| \\
&\le \frac{\csubL}{\k0 N^{1/2}}\left(\k0N\|v_n\|_X+|A\bar w|\right)\|w_n - \bar w\|\\
&\le \frac{\csubL}{\k0 N^{1/2}} \left( \k0N\|v\|_X+\Gamma \right)\|w_n - \bar w\| \;,
\end{align*}
\begin{align*}
|B(w_n-\bar w)| &\le \csubA |w_n - \bar w|^{1/2}|A(w_n - \bar w)|^{1/2} \|w_n - \bar w\|\\
&\le \frac{\csubA}{\k0 N} |A(w_n - \bar w)| \|w_n- \bar w\| \\
&\le \frac{2\csubA}{\k0 N}\Gamma \|w_n - \bar w\| \;,
\end{align*}
so that for each $t$
\begin{equation*}\label{3Best}
\int^s_{-\infty} e^{-\nu(s-\sigma)A}Q_N[B_{\text{sym}} (v_n+\bar w, w_n-\bar w)+ B(w_n-\bar w)] \ d \sigma \rightarrow 0 \quad \text{as} \ n \to \infty\;,
\end{equation*}
by the Lebesgue dominated convergence theorem. Indeed, the integrand is bounded by
\begin{equation}\label{bbb1}
2\tilde{\tilde{\Gamma}}\|w\|_Y e^{-\nu(s-\sigma)\k0^2}
\end{equation}
where
$$ \tilde{\tilde{\Gamma}} =(\csubA  N+\csubL N^{1/2})\|v\|_X + \left( \frac{\csubA}{\k0}+
 \frac{\csubL}{\k0 N^{1/2}} + \frac{2\csubA}{\k0 N} \right) \Gamma\;,
 $$
and goes to $0$ at each $\sigma$ since $\|w_n(t,\sigma)-\bar w(\sigma)\|\rightarrow 0$ when $n\rightarrow\infty$.

Let $k\in D(A)$. Note that the function $\sigma\mapsto  e^{-\nu(s-\sigma)A}k$ if $\sigma\le s$ and $\sigma\mapsto 0$ if $\sigma>s$ is in $L^1(\mathbb{R},H)$.
Thus by \eqref{wlim2}
\begin{equation}
\int_{-\infty}^s e^{-\nu(s-\sigma)A}(Q_N(B(v_n(t,\sigma)-\bar v(\sigma))+\bar R(t,\sigma)),k)\ d\sigma \to 0 \quad \text{as} \ n \to \infty\;.
\end{equation}
We also have
\begin{align*}
(\int_{-\infty}^s e^{-\nu(s-\sigma)A}
Q_N B(\bar v &+\bar w,v_n-\bar v) \ d \sigma,k)  \\
= &
\int_{-\infty}^s
Q_N (B(\bar v+\bar w, e^{-\nu(s-\sigma)A}k) ,v_n-\bar v) \ d \sigma \;.
\end{align*}
But by \eqref{A46b} we have that, as a function of $\sigma$
$$
B(\bar v(\sigma)+\bar w(\sigma), e^{-\nu(s-\sigma)A}k) \in L^{1}((-\infty,s],H)\;.
$$
We conclude that
\begin{equation*}\label{Bsym1}
\int_{-\infty}^s (e^{-\nu(s-\sigma)A}
Q_N B(\bar v +\bar w,v_n-\bar v) ,k)\ d \sigma \to 0 \quad \text{as} \ n \to \infty\;.
\end{equation*}
For the other half of $B_{\text{sym}}(\bar v+\bar w,v_n-\bar v)$ we write
$$
(e^{-\nu(s-\sigma)A}B(v_n-\bar v,\bar v+\bar w),k)=(v_n-\bar v,\theta(\sigma))
$$
where (due to \eqref{A46c}) we have
$$
|\theta(\sigma)| \le \|\bar v(\sigma)+\bar w(\sigma)\|  e^{-\nu(s-\sigma)\k0}|k|^{1/2}|Ak|^{1/2}  \quad  \forall \ \sigma \le s\;.
$$
It follows that $\theta(\sigma) \in L^{\infty}((-\infty,s])$, and hence
\begin{equation*}\label{Bsym1}
\int_{-\infty}^s (e^{-\nu(s-\sigma)A}
Q_N B(v_n-\bar v,\bar v +\bar w) ,k)\ d \sigma \to 0 \quad \text{as} \ n \to \infty\;.
\end{equation*}
In particular, for $k \in D(A)$, \eqref{wintform} now implies that
\begin{align*}\label{equal0w}
&(\bar w(s)-(\nu A)^{-1}h+\int_{-\infty}^s e^{-\nu(s-\sigma)A}Q_N(B(\bar v(\sigma)+\bar w(\sigma))-\bar R(t,\sigma)) d\sigma,k)=0
\end{align*}

Since $D(A)$ is dense in $H$, we deduce that for each $t$
\begin{equation}\label{wbarfinal}
\bar w(s)= (\nu A)^{-1}h-\int_{-\infty}^s e^{-\nu(s-\sigma)A}Q_N(B(\bar v(\sigma)+\bar w(\sigma))-\bar R(t,\sigma)) d\sigma \;.
\end{equation}
Note that \eqref{wbarfinal} implies that $Q_N\bar R(t,s)$ does not depend on $t$. Indeed, by the vector-valued version of the Lebesgue differentiation theorem we obtain that for each $t$ $Q_N\bar R(t,s)=r(s)$ for almost all $s$, fo some function
$r \in C(\mathbb{R},L^{\infty}(\mathbb{R},Q_NH)_{\text{weak}^*})$, which means that $Q_N\bar R$ is constant as a function in $C(\mathbb{R},L^{\infty}(\mathbb{R},H)_{\text{weak}^*})$.

Now observe that
\begin{align*}
v_n(t,s)&=v(t+t_n, s-t-t_n))\\
&=\int_{-\infty}^{t+t_n} e^{-\nu(t+t_n-\tau)A}
[g-P_NB(v(\tau,s-t-t_n)+w(\tau,s-t-t_n))]\ d\tau \\
&=\int_{-\infty}^t e^{-\nu(t-\tau)A}
[g-P_NB(v(\tau+t_n,s-t-t_n)+w(\tau+t_n,s-t-t_n))]\ d\tau \\
&=\int_{-\infty}^t e^{-\nu(t-\tau)A}
[g-P_NB(v_n(\tau,s+\tau-t)+w_n(\tau,s+\tau-t))]\ d\tau
\end{align*}
for all $t,s \in \bR$; consequently
\begin{equation}\label{vnstart}
\begin{aligned}
\int_{S_1}^{S_2} &(v_n(t,s),k)  \ ds = ((\nu A)^{-1} g, k)(S_2-S_1) \\
&-
\int_{-\infty}^t \left[ \int_{S_1}^{S_2} (B(v_n(\tau, s+\tau-t)+w_n(\tau, s+\tau-t)),
e^{-\nu(t-\tau)A}P_Nk)\ d s\right] \ d \tau \\
=&((\nu A)^{-1} g, k)(S_2-S_1) \\
&-
\int_{-\infty}^t \left[ \int_{S_1+\tau-t}^{S_2+\tau-t} (B(v_n(\tau, s)+w_n(\tau, s)),
e^{-\nu(t-\tau)A}P_Nk)\ d s\right] \ d \tau \;,&
\end{aligned}
\end{equation}
where $k \in D(A)$, $S_1 < S_2$ are fixed (but otherwise arbitrary).  By
the convergence in \eqref{wlim2} we can now pass to the limit in \eqref{vnstart}
by going through steps analogous to those from \eqref{wnstart} to \eqref{wbarfinal}.
Indeed, if $k$ is an eigenvector of $A$ with $Ak=\lambda k$, then by \eqref{Bexp}
\begin{equation*}
\begin{aligned}
\int_{-\infty}^t &\left[ \int_{S_1+\tau-t}^{S_2+\tau-t} (B(v_n(\tau, s)+w_n(\tau, s)),
e^{-\nu(t-\tau)A}P_Nk)\ d s\right] \ d \tau \\
=&\int_{-\infty}^t \left[ \int_{S_1+\tau-t}^{S_2+\tau-t} (B(\bar v+\bar w)-\bar R,
e^{-\nu(t-\tau)A}P_Nk)\ d s\right] \ d \tau\\
+&\int_{-\infty}^t \left[ \int_{S_1+\tau-t}^{S_2+\tau-t} (B_{\text{sym}}(v_n+\bar w,w_n-\bar w)+B(w_n-\bar w),
e^{-\nu(t-\tau)A}P_Nk)\ d s\right] \ d \tau\\
+&\int_{-\infty}^t  e^{-\nu(t-\tau)\lambda}\left[ \int_{S_1+\tau-t}^{S_2+\tau-t} (B_{\text{sym}}(\bar v+\bar w,v_n-\bar v), P_Nk)\ d s\right] \ d \tau \\
+&
\int_{-\infty}^t  e^{-\nu(t-\tau)\lambda}\left[ \int_{S_1+\tau-t}^{S_2+\tau-t} (\bar R+B(v_n-\bar v), P_Nk)\ d s\right] \ d \tau
\end{aligned}
\end{equation*}
On the right-hand side the second, third and fourth terms go to $0$ as $n\to \infty$. In the second term we use a bound similar to \eqref{bbb1} and the Lebesgue dominated convergence theorem. For the third and fourth note first that the integrands are bounded by $L^1$ functions, thus it suffices to show that for fixed $\tau$ the integrals
\[
\int_{S_1+\tau-t}^{S_2+\tau-t} (B_{\text{sym}}(\bar v+\bar w,v_n-\bar v),k) \ d s
\]
and
\[
\int_{S_1+\tau-t}^{S_2+\tau-t} (\bar R+B(v_n-\bar v), k)\ d s
\]
go to $0$ as $n\to \infty$. For the latter, note that the function $\chi_{[S_2+\tau-t,S_1+\tau-t](s)}k$ is in $L^1(\mathbb{R},H)$, where $\chi_I$ is the characteristic function of an interval $I$,  and use \eqref{wlim2}. For the former, write $(B_{\text{sym}}(\bar v+\bar w,v_n-\bar v),k)=(v_n(t,s)-\bar v(s),\psi(s))$ and note that, as above, by \eqref{A46b} and \eqref{A46c}  $\chi_{[S_2+\tau-t,S_1+\tau-t]}(s)\psi(s)$ is in $L^1(\mathbb{R},H)$ and therefore we obtain the convergence by \eqref{wlim1}.

Setting $\bar u=\bar v + \bar w$, we obtain
\begin{equation*}\label{vn2nd}
\begin{aligned}
&\int_{S_1}^{S_2} (\bar v(s),k)  \ ds -((\nu A)^{-1} g, k)(S_2-S_1) \\
&=-\int_{-\infty}^t \left[ \int_{S_1+\tau-t}^{S_2+\tau-t} (B(\bar u(s))-\bar R(\tau,s),e^{-\nu(t-\tau)A}P_Nk)  \ ds \right] \ d \tau \\
&=-\int_{-\infty}^t \left[ \int_{S_1}^{S_2} (B(\bar u(s+\tau-t))-\bar R(\tau,s+\tau-t),e^{-\nu(t-\tau)A}P_Nk)  \ ds \right] \ d \tau \\
&=-\int_{S_1}^{S_2}   (\int_{-\infty}^t e^{-\nu(t-\tau)A}P_N[B(\bar u(s+\tau-t))-\bar R(\tau,s+\tau-t) ] \ d \tau,k)  \ ds \\
&=-(\int_{S_1}^{S_2}   \left\{\int_{-\infty}^s e^{-\nu(s-\sigma)A}P_N[B(\bar u(\sigma))-\bar R(t-s+\sigma,\sigma) ] \ d \sigma \right\}   \ ds,k) \;.
\end{aligned}
\end{equation*}

Since for fixed $S_1 < S_2$, $t$ in $\bR$, $k$ can run over all
linear combinations of eigenvectors of $A$, of which are dense in $H$, we deduce
\begin{equation*}\label{vbarsemi}
\int_{S_1}^{S_2}\left\{ \bar v(s)-\int_{-\infty}^s e^{-\nu(s-\sigma)A}
[g-P_N B(\bar u(\sigma))+P_N\bar R(t-s+\sigma,\sigma)]\ d \sigma\right\} \ d s = 0
\end{equation*}
for all $S_1<S_2$, $t$ in $\bR$.

The final function in braces above is in $L^1([-S,S],H)$ for all $S>0$.
Therefore, the vector-valued version of Lebesgue's differentiation theorem now implies
\begin{equation}\label{Lebesgue}
\bar v(s)= \int_{-\infty}^s e^{-\nu(s-\sigma)A}[g-P_NB(\bar u(\sigma))+P_N\bar R(t-s+\sigma,\sigma) \ dt ]\ d \sigma
\end{equation}
 for $s\not\in\Sigma_t$, where the measure of $\Sigma_t$ is $0$.
Note that since  $C(\mathbb{R},L^\infty(\mathbb{R},H))_{\text{weak}^*}\subset L^1_{\text{loc}}(\mathbb{R}^2)$, both the right- and left-hand sides of \eqref{Lebesgue} are in $L^1_{\text{loc}}(\mathbb{R}^2)$, and therefore by Fubini's theorem \eqref{Lebesgue} holds for almost all pairs $(t,s)\in\mathbb{R}^2$. By substituting $(t,s)$ with $(t+\tau,s+\tau)$ we obtain that for almost all $t,s$
\begin{equation}\label{Lebesgue01}
\bar v(s+\tau)= \int_{-\infty}^{s+\tau} e^{-\nu(s+\tau-\sigma)A}[g-P_NB(\bar u(\sigma))+ P_N\bar R(t-s+\sigma,\sigma) \ dt ]\ d \sigma
\end{equation}
holds for almost all $\tau$. Fix such $t,s$ that \eqref{Lebesgue01} holds for almost all $\tau$. Then the right-hand side of \eqref{Lebesgue01} is continuous in $\tau$ and hence $\bar v$ coincides almost everywhere with continuous function. Thus, we can suppose that $\bar v$ is (absolutely) continuous. Taking the derivative with respect to $\tau$ of \eqref{Lebesgue01}, we obtain that for almost all $(t,s)\in\mathbb{R}^2$,
\begin{equation*}\label{vbarODE01}
\frac{d \bar v(s+\tau)}{ds} + \nu A \bar v(s+\tau) + P_NB(\bar u(s+\tau),\bar u(s+\tau))= g +P_N{\bar R}(t+\tau,s+\tau)
\;,
\end{equation*}
for almost all $\tau$. It follows that
\begin{equation*}\label{vbarODE1}
\frac{d \bar v(s)}{ds} + \nu A \bar v(s) + P_NB(\bar u(s),\bar u(s))= g +P_N{\bar R}(t,s)
\;,
\end{equation*}
for almost all $t,s$. We therefore obtain that the function $\bar R\in C(\mathbb{R},L^\infty(\mathbb{R},H)_{\text{weak}^*})$ is almost everywhere equal to some fixed element of $L^\infty(\mathbb{R},H)$. Since $\bar R$ is continuous, this equality holds everywhere, thus $\bar R$ does not depend on $t$.
\end{proof}

\section{Stationary solutions of \eqref{ode_final}} \label{statsec}
Recall that $u_0\in H$ is stationary for \eqref{NSE} if
\[
\nu{}Au_0+B(u_0,u_0)=f.
\]
Clearly, if
\begin{equation}\label{statproj}
v(0,\cdot)=P_Nu_0(\cdot)\;,
\end{equation}
then $v(t,\cdot)=P_Nu_0(\cdot)$ for all $t$ is a stationary
solution of \eqref{ode_final}.  In fact by Theorem \ref{inv} any stationary solution $v(t,s)$ of
\eqref{ode_final} which is constant in $s$ must satisfy \eqref{statproj}.
The set of all stationary solutions of \eqref{ode_final} may, however,  be much larger
than the set of constant functions in $X$ that are the projections of the stationary solutions of \eqref{NSE} as in \eqref{statproj}.  In fact we show in \cite{FJKr2} that a determining form for the Lorenz
system has stationary solutions which are not in the global attractor of the Lorenz system.

  A solution $v(t,\cdot)$ is a stationary
solution of
$\eqref{ode_final}$ if and only if $v(t,s)=v_0(s)$ for all $t$ and
\[
\nu{}Av_0(s)=g-P_NB(\varphi v_0(s)+W(v_0)(s)),
\]
that is, if and only if it satisfies the following differential algebraic equation
\begin{equation}
\label{eq1}
\nu{}Av_0(s)=g-P_NB(\varphi v_0(s)+w_0(s)),
\end{equation}
\begin{equation}
\label{eq2}
\frac{dw_0(s)}{ds}+\nu{}Aw_0(s)=h-Q_NB(\varphi v_0(s)+w_0(s))\;,
\end{equation}
where $w_0=W(v_0)$.
In this case the energy and enstrophy balances are
\begin{equation}
\label{energy2}
\frac{1}{2}\frac{d}{ds}|w_0|^2+\nu\|u_0\|^2=(f,u_0),
\end{equation}
\begin{equation}
\label{enstrophy2}
\frac{1}{2}\frac{d}{ds}\|w_0\|^2+\nu|Au_0|^2=(f,Au_0),
\end{equation}
where $u_0=v_0+w_0$  (compare to  \eqref{energyeq}, \eqref{enstrophyeq}).

It is shown in  \cite{DFJ} that solutions on the global attractor of the
NSE must satisfy $\nu\|u(t)\|^2\leq{}|f||u(t)|$ for all
$t\in\mathbb{R}$, with equality holding for some $t$ if and only
if $f$ is an eigenvector of $A$.  One might ask if at least the stationary
solutions of \eqref{ode_final} satisfy an analogous property.
We will give below a partial answer for this case, which involves an
application of \eqref{energy2}, \eqref{enstrophy2}.

Let $u_0=v_0+w_0$. We first prove the following:

\begin{lemma}
\label{lemma8.1}
Let $w_0$ and $v_0$ satisfy \eqref{eq1} and \eqref{eq2}, and let
$\sigma_0<\sigma_1$.  Suppose that $|w_0(s)|$ is nondecreasing
in $(\sigma_0,\sigma_1)$, then
%\begin{itemize}
\begin{enumerate}
\renewcommand{\theenumi}{(\roman{enumi})}
\renewcommand{\labelenumi}{\theenumi}
\item $\nu\|u_0(s)\|^2\leq{}|f||u_0(s)|$ for all $s\in(\sigma_0,\sigma_1)$.
\item If, in addition, $w_0(s)$ is nonzero for all $s\in(\sigma_0,\sigma_1)$, $\|w_0(s)\|^2/|w_0(s)|$ is nondecreasing, and for some
$s_0\in(\sigma_0,\sigma_1)$, we have $\|u_0(s_0)\|^2/|u_0(s_0)|=|f|/\nu$,
 then $f$ is an eigenvector of $A$ for the eigenvalue
$\lambda^0=\|u_0(s_0)\|^2/|u_0(s_0)|^2$, $\lambda^0>(N\kappa_0)^2$,  and
$f=Q_Nf$.
\end{enumerate}
%\end{itemize}
\end{lemma}
\begin{proof}
In the proof we drop the subscript on $w_0,v_0,u_0$, to simplify notation.
 Denote $\chi_w=\|w\|^2/|w|$, $\chi_u=\|u\|^2/|u|$,
$\lambda_w=\|w\|^2/|w|^2$, $\lambda_u=\|u\|^2/|u|^2$.
Since $w(s)\neq{}0$ in $(\sigma_0,\sigma_1)$, these quantities are well-defined.
From \eqref{energy2} and the fact that $|w|$ is nondecreasing it follows that
$(f,u)-\nu\|u\|^2\geq{0}$, which in turn implies that
$\chi_u(s)\leq{}|f|/\nu$, proving (i).

Passing to the proof of (ii), we start by observing that from \eqref{energy2},
\eqref{enstrophy2} it follows
that
\begin{equation}\label{2a}
\begin{aligned}
&\chi_w(s)-\chi_w(s_0)=\\&\int^s_{s_0}\frac{2\nu}{|w|}\left[|\frac{f}{2\nu}|^2-\frac{\chi^2_u}{4}-|\frac{f}{2\nu}-(A-\frac{\lambda_u}{2})u|^2-\frac{\lambda_w-\lambda_u}{2}((\frac{f}{\nu},u)-\|u\|^2)\right]d\sigma.
\end{aligned}
\end{equation}
By (i), since $|w|$ is nondecreasing, and  the assumption $\chi_u(s_0)=|f|/\nu$,
we deduce from \eqref{energy2} that $(f,u)=|f||u|$; hence
we have $u(s_0)=cf$, where $c=|f|^2/\nu\|f\|^2$.

Since $\chi_w$ is assumed to be nondecreasing, the integrand in
\eqref{2a} is nonnegative for almost all $\sigma$; also every
summand in it is continuous, except for $|f/2\nu-(A-\lambda_u/2)u|$, which is
lower semicontinuous.  This suffices however, to ensure that the integrand
is in fact nonnegative at every point of $(\sigma_0,\sigma_1)$.

In particular, at the point $s_0$ we have that $\|u\|^2=(f/\nu,u)$, thus it
follows that
    \begin{equation}
\label{12}
   |\frac{f}{2\nu}-\left(A-{\lambda_{u(s_0)}\over 2}\right)u(s_0)|=0,
    \end{equation}
hence $(A-\lambda^0)f=0$, where $\lambda^0=\lambda_{u(s_0)}$.  Indeed,
\begin{equation*}
%\label{eigenvalue}
\begin{aligned}
&(A-\lambda^0)f=(A-\frac{\lambda^0}{2})f-\frac{\lambda^0}{2}f=\frac{\nu\|f\|^2}{|f|^2}(A-\frac{\lambda^0}{2})u(s_0)-\frac{\lambda^0}{2}f=\frac{\nu\|f\|^2}{|f|^2}\frac{f}{2\nu}-\frac{\lambda^0}{2}f=0.
\end{aligned}
\end{equation*}
 If $\lambda^0\leq{}(N\kappa_0)^2$ it would follow that $f=g$, which gives that
$w(s_0)=0$, but this is contrary to our assumption.  Thus we have that
$\lambda^0>(N\kappa_0)^2$, and so $f=h$, $g=0$, $v(s_0)=0$,
$w(s_0)=h/\nu\lambda^0$ and $Ah=\lambda^0h$.
\end{proof}
Let $h_0=h/\nu\lambda^0$. Now we need the following lemma:
\begin{lemma}
\label{lem222}
Assume that $v_0$ and $w_0$ satisfy \eqref{eq1} and \eqref{eq2}, and that $g=0$
and $h$ is an eigenvector of $A$.  Then for large enough $N$, $v(s)=0$ and $w(s)=h_0$ for
all $s$, where $\lambda^0$ is an eigenvalue of $A$.
\end{lemma}
\begin{proof}
Once again, denote $w_0,v_0, u_0$ by $w,v,u$ for simplicity.
From \eqref{eq1}, \eqref{agmon_1}, and \eqref{3.19.5} we obtain that
\begin{equation*}
\begin{aligned}
\label{}
\nu\|v\|^2&=-(B(u,w),v)=-(B(w,w),v)+(B(v,v),w)\\
&\leq{}\-(B(w,w),v)+c_T\|v\|^2\|w\|\frac{(\ln{eN})^{1/2}}{\kappa_0N},
\end{aligned}
\end{equation*}
whence
\begin{equation}\label{eqq1}
\nu\|v\|^2\left[1-c_T\frac{G\beta(G,N)}{N}(\ln{eN})^{1/2}\right]\leq{}-(B(w,w),v).
\end{equation}

 By denoting $w=h_0+\delta$ we have
\begin{equation}\label{eqq2}
\begin{aligned}
-(B(w,w),v)=&-(B(\delta,w),v)-(B(h_0,\delta),v)\\
\leq{}&(\ln{eN})^{1/2}\|\delta\|(c_T|w|\|v\|+c_B|h_0|\|v\|)\\
\leq{}&(\ln{eN})^{1/2}(c_T\frac{\nu{G}\beta(G,N)}{N}+c_B\nu\frac{\kappa_0^2G}{\lambda^0})\|v\|\|\delta\|.
\end{aligned}
\end{equation}

From \eqref{eqq1} and \eqref{eqq2} we obtain
\begin{equation}
\label{vdelta}
\|v\|\left[1-c_T\frac{G\beta(G,N)}{N}(\ln{eN})^{1/2}\right]\leq{}(\ln{eN})^{1/2}\frac{G}{N}(c_T\beta(G,N)+c_B\frac{\kappa_0^2N}{\lambda^0})\|\delta\|.
\end{equation}

 Therefore if
 \begin{equation}
\label{lastbound1}
N>2c_TG\beta(G,N)(\ln{eN})^{1/2},
 \end{equation}
we have the bound
\begin{equation}
\label{vdelta1}
\|v\|\leq{}c_1\|\delta\|,
\end{equation}
where
\[
c_1=2(\ln{eN})^{1/2}\frac{G}{N}(c_T\beta(G,N)+c_B\frac{1}{N}).
\]
Now using \eqref{eq2} and \eqref{vdelta1} we obtain that
\begin{equation*}
\label{eqdelta}
\begin{aligned}
\frac{1}{2}&\frac{d}{ds}|\delta|^2+\nu\|\delta\|^2=-(B(u,u),\delta)=-(B(v+h_0+\delta,v+h_0),\delta)\\
&=-(B(v+\delta,h_0),\delta)-(B(u,v),\delta)\\
&\leq{} c_T(\ln{eN})^{1/2}\|v\|\|h_0\||\delta|+\csubL{}|\delta|\|\delta\|\|h_0\|+c_T(\ln{eN})^{1/2}\|u\|\|v\||\delta| \\
&\leq{}\nu(c_Tc_1(\ln{eN})^{1/2}+\csubL{})\frac{G}{N^2}\|\delta\|^2+c_Tc_1(\ln eN)^{1/2}\frac{\|v\|_X+\nu\kappa_0G\beta(G,N)}{\kappa_0N}\|\delta\|^2 \;.
\end{aligned}
\end{equation*}

It follows from Lemma~\ref{useit} and Theorem~\ref{absorbing_ball1}, that if $N$ is such that \eqref{lastbound1}
holds and moreover
\begin{equation}
\label{lastbound2}
\begin{aligned}
&N^2>2G(c_Tc_1(\ln{eN})^{1/2}+\csubL{}),\\&
N>2c_Tc_1(G+c_T(\ln{eN})^{1/2}G^2(6+\beta(G,N))\beta(G,N)+G\beta(G,N)),
\end{aligned}
\end{equation}
then $\delta(s)=0$ for all $s$. Therefore from \eqref{vdelta1} we have for all $s$ that
$v(s)=0$ , and hence $u(s)=h/\nu\lambda^0=w(s)$.   Since $h$ is an eigenvector
of $A$, it follows from   \eqref{energy2} that $Ah=\lambda^0 h$
\end{proof}

Therefore we obtain the following proposition.
\begin{proposition}\label{prop111}
Suppose $N$ satisfies \eqref{lastbound1} and \eqref{lastbound2}.
 If the assumptions of  Lemma~\ref{lemma8.1} hold for $w_0=W(v_0)$ for some stationary solution
$v_0$ of \eqref{ode_final}, then any
stationary solution of \eqref{ode_final} is zero.
\end{proposition}

We finish by simplifying the conditions on $N$ in Proposition \ref{prop111}.  Suppose that $G\gg1$ and
\begin{equation}
\label{lb1}
N>CG(\ln{eN})^{1/2},
\end{equation}
for some sufficiently large constant $C$.  Then \eqref{lastbound1} holds and
moreover $\beta(G,N)\leq{}36c_T/C$, and so $\beta(G,N)$ is small provided that $C$ is large
enough. Choose $C$ such that $\beta(G,N)<1$.  Then $c_1\leq{}2(c_T+c_B)/C$.
Choose $C$, so that also $c_1<1$.  Moreover, the bound in
Theorem~\ref{absorbing_ball1} is less than
$\nu\kappa_0(G+2c_TG^2(\ln{eN})^{1/2})\leq{}C_2G^{2}(\ln{eN})^{1/2}$ for some
constant $C_2$, which does not depend on $C$, since $G>>1$.  We obtain then that the
first inequality in \eqref{lastbound2} holds, and to satisfy the second
inequality in \eqref{lastbound2}, it suffices to take
\begin{equation}
\label{finalbound}
N>\varepsilon{}G^2(\ln{eN})^{1/2},
\end{equation}
 for any fixed $\varepsilon$ (provided that $G$ is large). Note that if
\eqref{finalbound} holds then \eqref{lb1} also holds.  It follows that
conclusion of Proposition~\ref{prop111} is true as long as inequality
\eqref{finalbound} holds for $N$.

\section{Another approach}
We present now in a concise manner an approach
producing an ODE similar to \eqref{ode_final} for the determining modes $P_Nu$, which however can also be obtained for other determining parameters
such as nodal values, finite elements, local spatial averages and interpolants on fine enough spatial grids (see, e.g.,  \cite{CJTi1,FT1, FT2, JT93}, for works related to various determining parameters).
We first outline this approach for the case of determining modes
so it can be compared to the first approach.  Then we will state some
results for other interpolants.
This unified approach is fully developed in \cite{FJKrT2}.

Let $u(\cdot)$ be a trajectory  in $\cA$, and suppose $w$ satisfies the equation
\begin{equation}\label{wteqn}
\frac{dw}{ds}+\nu Aw+B(w,w)=f-\nu\k0^2\mu P_N (w-u)\;,
\end{equation}
for $s > s_0$ , with $w(s_0) \in H$. Here $\mu >0$ is a damping dimensionless parameter to be determined later in terms of the Grashof number $G$.

The difference $\delta=w-u$ then satisfies
$$
\frac{d \delta}{ds} + \nu A\delta +B(\delta,u)+B(u,\delta)+B(\delta,\delta)=-\nu\k0^2\mu P_N \delta
$$
and hence
\begin{align*}
\frac{1}{2}\frac{d }{ds}|\delta|^2 + \nu \|\delta\|^2 + (B(\delta, u),\delta) &= -\nu\k0^2\mu |P_N \delta|^2  \\
   &= -\nu\k0^2\mu |\delta|^2+\nu\k0^2\mu |Q_N \delta|^2 \\
   &\le -\nu\k0^2\mu |\delta|^2+\frac{\nu\mu}{(N+1)^2} \|Q_N \delta\|^2\;.
\end{align*}
Now take $\mu$ large enough so that
\begin{equation}\label{mucond2}
\alpha=1 -2\cL^2G^2 +2\mu > 0\;,
\end{equation}
and then $N$, large enough, so that
\begin{equation}\label{mucond}
{\mu \over (N+1)^2} \le {1\over 4}\;.
\end{equation}
We then have that
\begin{align*}
\frac{1}{2}\frac{d }{ds}|\delta|^2 +\nu \|\delta\|^2 &\le  \cL |\delta|\|\delta\|\|u\| -\nu\k0^2\mu |\delta|^2
+\frac{\nu\mu}{(N+1)^2} \|\delta\|^2 \\
   &\le \frac{\nu}{4} \|\delta\|^2+ \frac{\cL^2}{\nu}|\delta|^2\|u\|^2-\nu\k0^2\mu | \delta|^2
    +  \frac{\nu}{4} \|\delta\|^2\\
   &\le \frac{\nu}{2} \|\delta\|^2+ \cL^2\nu\k0^2G^2|\delta|^2-\nu\k0^2\mu | \delta|^2   \end{align*}
so that
\begin{align*}
\frac{d }{ds}|\delta|^2 +\alpha \nu\k0^2|\delta|^2 \le 0\;,
\end{align*}
and
$$
|\delta(s)| \le |\delta(s_0)|e^{-\alpha\nu\k0^2(s-s_0)}\;.
$$
If the solution $w(s)$ to equation \eqref{wteqn} is defined for all
$s \in [s_0, \infty)$, then $\delta(s)\to 0$ as $s \to \infty$.
Moreover, if $w$ is defined, and bounded on $(-\infty, \infty)$, then
$\delta(s)=0$ for all $s \in \bR$.  Thus, if we first choose $\mu$ to satisfy \eqref{mucond2}, and then
$N$ to satisfy \eqref{mucond}, the only bounded solution to \eqref{wteqn}
on $(-\infty, \infty)$ is $w=u$.

The analog $\tilde W$, of the map $W$ defined in Section \ref{secExtend} is now given as follows: let $v\in C_b(\mathbb{R},P_NV)$ and let
$w=\tilde{W}(v)$ be the unique solution of the equation
\begin{equation}\label{wtteqn}
\frac{dw}{ds}+\nu Aw+B(w)=f-\nu\k0^2\mu P_N (w-v)
\end{equation}
that is  bounded on $\bR$.  This solution
can be constructed using a Galerkin approximation.
The uniqueness of $w$ can be proven by the same method as in the proof of Lemma~\ref{3.1};
To obtain an a priori estimate, first note that
%\begin{equation*}
%\|w(s)\|^2\le\frac{G^2\nu\k0^2}{N+1}+\frac{N}{N+1}\sup_{s}\|v(s)\|^2 \;.
%\end{equation*}
\begin{align*}
\frac{1}{2}\frac{d }{ds}\|w\|^2 + \nu |Aw|^2 =(f,Aw) -\nu\k0^2\mu |P_N Aw|^2 +\nu\k0^2\mu(v,Aw)
\end{align*}
so
\begin{align*}
\frac{d }{ds}\|w\|^2 + 2\nu |Aw|^2 \le
2(|f|+\nu\k0\mu\|v\|_X)|Aw| \;,
\end{align*}
%$$
%\|w\|^2 \le \left[\k0^2\min(\lambda_{N+1}^2, 1+\mu/\nu)\right]^{-1}
%\left(|Aw|^2+\frac{\mu}{\nu} |P_N Aw|^2\right)\;,
%$$
and we have by Lemma \ref{useit}
$$
\|w\|  \le
\frac{|f|}{\nu\k0}+\mu\|v\|_X\;.
$$

It is easy to show that if $v(\cdot)=P_Nu(\cdot)$ for a trajectory
$u(\cdot)$ in $\cA$,  then $\tilde{W}(v)=u$.  This suggests that
we study the following analog of  \eqref{ode_final}
\begin{equation}\label{gam}
\frac{dv(t)}{dt}=P_Nf -\nu Av(t)-P_NB(\tilde{W}(v(t)))%- \nu\k0^2\mu P_N(\tilde{W}(v(t))-v(t)).
\;.
\end{equation}
The fact that for large $\mu$ the map $\tilde{W}:v\mapsto \tilde{W}(v)$ is
Lipschitz is again obtained by a method similar to the one used in the proof
of Corollary~\ref{w_lipschitz}, and that the right-hand side of \eqref{gam} is
locally Lipschitz as a map from the space $C_b(\mathbb{R},P_NV)$ to itself is
proved similarly. With these considerations the differential equation
\eqref{gam} is an ODE in the space $C_b(\mathbb{R},P_NV)$.  Note that the ODE \eqref{gam}
depends not only on $N$, but also on the parameter  $\mu$ through the function $\tilde W$.

As in Section 7, translation invariant solutions of
\eqref{gam} are exactly those that come from the NSE.
For example, if $v$ is translation invariant, then
$$v(t)(s)=v(t+s)(0)=v_0(t+s)\quad \text{and} \quad w(t)(s)=w(t+s)(0)=w_0(t+s)\;.
$$  Moreover,
since
$$
\frac{dv_0}{ds}=\frac{dv_0}{dt} \quad \text{at} \ s+t\;,
$$
we rewrite \eqref{gam} as
\begin{equation}\label{gam2}
\frac{dv}{ds}+\nu Av+P_NB(w)=P_Nf%- \nu\k0^2\mu P_N(w-v)
\;,
\end{equation}
and subtract $P_N$ applied to \eqref{wteqn}
\begin{equation*}
\frac{dP_N w}{ds}+\nu AP_Nw+P_NB(w)=P_Nf - \nu\k0^2\mu P_N(w-v)
\end{equation*}
to obtain
\begin{equation}\label{xieq}
\frac{d\xi}{ds}+\nu A \xi=\nu\k0^2\mu \xi\;,
\end{equation}
where $\xi=v-P_Nw=P_N(v-w)$.   Since $\xi(s)$ is defined and bounded
for all $s \in \bR$, and \eqref{xieq} is linear, we have $\xi(s)=0$ for all
$s \in \bR$, provided that $\kappa_0^2\mu$ is not an eigenvalue of $A$, which can always be achieved.  As a consequence \eqref{wteqn} reduces to the NSE.

 Note that we do not
have to prepare the nonlinear term as was done in \eqref{ode_final}
with the cut-off function $\varphi$ in order to establish these properties.  Preparation of the
nonlinear term is, however, needed if we want to prove the existence
of absorbing ball for \eqref{gam}.

One advantage that \eqref{gam} has over \eqref{ode_final} comes in the study of
stationary
solutions.  The analog for \eqref{gam} of the DAE in Section \ref{statsec} is
\begin{align*}
\nu Av+P_NB(w)&=P_Nf \\%- \nu\k0^2\mu P_N(w-v) \\
\frac{dw}{ds}+\nu Aw+B(w)&=f-\nu\k0^2\mu P_N (w-v)\;.
\end{align*}
Note that the algebraic relation can be solved explicitly
$$
v=\nu^{-1}A^{-1}P_N(f-B(w))\;.
$$
 Thus in contrast with \eqref{ODE},  stationary solutions of \eqref{gam}
 are directly associated with the globally bounded solutions of an explicit Navier-Stokes type equation in $H$, namely
\begin{equation}\label{statEdriss}
\frac{dw}{ds}+\nu[A+\k0^2\mu P_N]w+
[I + \k0^2\mu P_NA^{-1}P_N ]B(w)=[I + \k0^2\mu P_NA^{-1}P_N]f \;.
\end{equation}
The study of the equations \eqref{gam} and \eqref{statEdriss} for various types of determining parameters will be presented in \cite{FJKrT2}.

In \cite{FJKrT2} we also study a larger class of determining parameters. In particular as those described below.
Let $J$ be a bounded linear operator on $H$ (approximating the identity)
satisfying the following conditions:
\begin{enumerate}[(i)]\label{listl}
	\item\label{onec} 	 $|\phi - J(\phi)| \le  c_I h^2 |A \phi |$, for every $\phi \in D(A)$, where  $h$ is the size of the grid; moreover,  $h^2$ is required to be at most $O((\kappa_0^2 G \log (1+G))^{-1})$, i.e. $N=(\kappa_0 h)^{-2}$ is of the order of $O(G \log (1+G))$.
	\item\label{twoc} $J H\subset D(A^2)$
	\item\label{threec} there exists a linear bounded operator $K$ on $H$ such that $J A=K$ on $D(A)$.
\end{enumerate}
Such $J(u)$ is a determining system of parameters for \eqref{NSE}. This is proven as before by using the equation
%Let $u(\cdot) \subset \cA$. Suppose $|w(s)|$ is bounded on $\bR$ and satisfies
\begin{equation}\label{wIeq1}
\frac{dw}{ds}+\nu Aw+B(w,w)=f-\mu\nu\k0^2 (J (w)-J(u)),
\end{equation}
where $\mu$ is in the appropriate interval.
%$\mu\in (1/(\cL^2G^2),1/\eta^2)$.

The map $W$ is defined using this equation instead of \eqref{wteqn}.
Now the equation which replaces the equation \eqref{gam} is the following
\begin{equation}\label{11.10}
\frac{d}{dt}v(t)+\nu J A W(v(t))-\nu\kappa_0^{-2}  P_{JH}A^2(JW(v(t))-v(t))  +J B(W(v(t)))=J f
\end{equation}
which is an ordinary differential equation in the ball $B(0,R(\eta,\mu))$ of the space $X=C_b(\mathbb{R},J(H))$. In \eqref{11.10} $P_{JH}$ denotes the orthogonal projection of $H$ onto the closure of the  range of $J$.

We conclude this summary by observing that given a finite rank linear operator $J$ on $H$ satisfying the condition \eqref{onec}, then for every sufficiently small $\varepsilon>0$ there exists a $J'$ satisfying all conditions \eqref{onec}-\eqref{threec} and such that the norm of $J-J'$ as an operator on $H$ is less then $\varepsilon$.
\vskip .2truein
\centerline{{\bf Acknowledgments}}

This paper is dedicated to Professor Peter Constantin, on the occasion of his  60th birthday, as token of friendship and admiration for his contributions to research in partial differential equations and fluid mechanics.

The work of C.F. is supported in part by NSF grant DMS-1109784, that of M.J.
by  DMS-1008661 and DMS-1109638, that of R. K. by the ERC starting grant GA 257110 ÒRaWGÓ, and that of E.S.T.  by DMS-1009950, DMS-1109640 and DMS-1109645, as well as the Minerva Stiftung/Foundation.
\bibliographystyle{acm}
%\bibliography{net_new1}
\bibliography{FJKrTi}

\end{document}